\documentclass[11pt]{article}

\usepackage[english]{babel} 
\usepackage{microtype}
\usepackage{mathtools}
\usepackage{amssymb}
\usepackage{amsfonts}
\usepackage{graphicx}
\usepackage[usenames,dvipsnames,svgnames,table]{xcolor}
\usepackage{float}
\usepackage{bm}
\usepackage{bbm}
\usepackage{tikz}
\usetikzlibrary{arrows.meta,decorations.pathmorphing,backgrounds,positioning,fit,petri,intersections,calc,through,shapes.misc,graphs,lindenmayersystems}
\usepackage{amsthm}
\usepackage[symbol]{footmisc}

\usepackage[bookmarks=true]{hyperref}
\usepackage{bookmark}
\usepackage{import}
\usepackage{adjustbox}
\usepackage{paralist, tabularx}

\numberwithin{equation}{section}
\theoremstyle{plain}
\newtheorem{Stelling}{Theorem}[section]
\newtheorem{Lemma}[Stelling]{Lemma}
\newtheorem{propositie}[Stelling]{Proposition}
\newtheorem{Corollary}[Stelling]{Corollary}
\theoremstyle{definition}
\newtheorem{Def}[Stelling]{Definition}

\newtheorem{Remark}[Stelling]{Remark}

\providecommand{\eps}{\varepsilon}

\newcommand{\E}{\mathbb{E}}

\newcommand{\id}{\operatorname{id}}

\newcommand{\Unif}{\operatorname{Unif}}

\newcommand{\MLE}{\operatorname{MLE}}

\newcommand{\KL}{\operatorname{KL}}

\newcommand{\mA}{\mathcal{A}}

\newcommand{\mD}{\mathcal{D}}

\newcommand{\mF}{\mathcal{F}}
\newcommand{\mG}{\mathcal{G}}
\newcommand{\mH}{\mathcal{H}}

\newcommand{\mS}{\mathcal{S}}

\renewcommand{\bf}{\mathbf{f}}
\newcommand{\bg}{\mathbf{g}}
\newcommand{\bp}{\mathbf{p}}
\newcommand{\bq}{\mathbf{q}}
\newcommand{\bv}{\mathbf{v}}

\newcommand{\bfm}{\mathbf{m}}
\newcommand{\bG}{\mathbf{G}}
\newcommand{\bu}{\mathbf{u}}
\newcommand{\bx}{\mathbf{x}}
\newcommand{\bX}{\mathbf{X}}
\newcommand{\wbX}{\widetilde{\mathbf{X}}}
\newcommand{\by}{\mathbf{y}}
\newcommand{\bY}{\mathbf{Y}}
\newcommand{\wbY}{\widetilde{\mathbf{Y}}}

\newcommand{\bPhi}{\bm{\Phi}}
\newcommand{\bpsi}{\bm{\psi}}

\newcommand{\bgamma}{\bm{\gamma}}

\newcommand{\cmmnt}[1]{\ignorespaces}
\newcommand{\1}{\mathbbm{1}}

\DeclarePairedDelimiter\ceil{\lceil}{\rceil}
\DeclarePairedDelimiter\floor{\lfloor}{\rfloor}

\tikzset{InterPoint/.style={black!50}}
\tikzset{Fr/.style={thick, black}}
\tikzset{Hr/.style={thick, black!40!white}}
\tikzset{help lines/.style={very thin, black!20,step=0.25}}
\tikzset{YAxis/.style={->}}
\tikzset{XAxis/.style={->}}
\newcommand\Square[1]{+(-#1,-#1) rectangle +(#1,#1)}

\begin{document}

\title{Convergence rates of deep ReLU networks for multiclass classification}
\author{%
    Thijs Bos\footnotemark[1] \footnotemark[3]%
    \ \ and Johannes Schmidt-Hieber\footnotemark[2] \footnotemark[3]%
}
  \footnotetext[1]{Leiden University} 
  \footnotetext[2]{University of Twente}
  \footnotetext[3]{The research has been supported by the NWO/STAR grant 613.009.034b and the NWO Vidi grant VI.Vidi.192.021.}
\date{}
\maketitle

\begin{abstract}

For classification problems, trained deep neural networks return probabilities of class memberships. In this work we study convergence of the learned probabilities to the true conditional class probabilities. More specifically we consider sparse deep ReLU network reconstructions minimizing cross-entropy loss in the multiclass classification setup. Interesting phenomena occur when the class membership probabilities are close to zero. Convergence rates are derived that depend on the near-zero behaviour via a margin-type condition.
\end{abstract}



\section{Introduction}
The classification performance of a procedure is often evaluated by considering the percentage of test samples that is assigned to the correct class. The corresponding loss for this performance criterion is called the $0$-$1$ loss. Theoretical results for this loss are often related to the the margin condition \cite{mammen1999,tsybakov2004optimal,PlugInRatesMarginCondition}, which allows for fast convergence rates. Empirical risk minimization with respect to the non-convex $0$-$1$ loss is computationally hard and convex surrogate losses are used instead, see for example \cite{bartlett2006convexity,TariganGeerSVMMultiHinge}. More recently, similar results have been obtained for deep neural networks. This includes results for standard deep neural networks in combination with the hinge and logistic loss as surrogate losses \cite{kim2018fast}, as well as results for deep convolutional neural networks with the $L_2$ loss \cite{kohler2020rate} and  logistic loss \cite{kohler2020statistical} as surrogate losses.

Trained neural networks provide more information than just a guess of the class membership. For each class and each input, they return an estimate for the probability that the true label is in this class. For an illustration, see for example Figure 4 in the seminal work \cite{ImagenetClassification}. In applications it is often important how certain a network is about class memberships, especially in safety-critical systems where a wrong decision can have serious consequences such as automated driving \cite{SelfDrivingEndToEnd} and AI based disease detection \cite{leibig2017leveraging,DeepLearningMedicalImaging}. In fact, the conditional class probabilities provide us with a notion of confidence. If the probability of the largest class is nearly one, it is likely that this class is indeed the true one. On the other hand, if there is no clear largest class and the conditional class probabilities of several classes are close to each other, it might be advisable to let a human examine the case instead of basing the decision only on the outcome of the algorithm.

To evaluate how fast the estimated conditional class probabilities of deep ReLU networks approach the true conditional class probabilities, we consider in this work convergence with respect to the cross-entropy (CE) loss. If the conditional class probabilities are bounded away from zero or one, the problem is related to regression and density estimation. Therefore, it seems that one could simply modify the existing proofs on convergence rates for deep ReLU networks in the regression context under the least squares loss \cite{NonParametricRegressionReLU,bauer2019}. This does, however, not work since the behaviour of the CE loss differs fundamentally from that of the least squares loss for small conditional class probabilities. The risk associated with the CE loss is the expectation with respect to the input distribution of the Kullback-Leibler divergence of the conditional class probabilities. If an estimator becomes zero for one of the conditional class probabilities while the underlying conditional class probability is positive, the risk can even become infinite, see Section \ref{sec.StatModel}. In many applications where deep learning is state-of-the-art, the covariates contain nearly all information about the label and hence the conditional class probabilities are close to zero or one. For example in image classification it is often clear which object is shown on a picture. To deal with the behaviour near zero, we introduce a truncation of the CE loss function. This allows us to obtain convergence rates without bounding either the true underlying conditional class probabilities or the estimators away from zero. Instead our rates depend on an index quantifying the behaviour of the conditional class probabilities near zero. Convergence rates and the condition on the conditional class probabilities can be found in Section \ref{sec.results}.

\textit{Notation:} We denote vectors and vector valued functions by bold letters. For two vector valued functions $\bf=(f_1, \ldots, f_d)$ and $\bg=(g_1, \ldots, g_d)$ mapping $\mD$ to $\mathbb{R}^d,$ we set
$\|\bf-\bg\|_{\mD,\infty} := \big\| \max_{j=1,\ldots, d}|f_j(\bx)-g_j(\bx)|\big\|_{L^\infty(\mD)}.$
If it is clear to which domain $\mD$ we refer to, we also simply write $\|\bf-\bg\|_{\infty}.$
For a vector $\bv=(v_1,\dots,v_m)$ and a matrix $W=(W_{i,j})_{i=1,\dots,n;j=1,\dots,m}$ we define the maximum entry norms as $\|\bv\|_{\infty}:=\max_{i=1,\dots,m}|v_i|$ and $\|W\|_{\infty}:=\max_{i=1,\dots,n}\max_{j=1,\dots,m}|W_{i,j}|.$ The counting `norm' $\|\bv\|_{0}$, $\|W\|_{0}$ is the number of nonzero entries in the vector $\bv$ and matrix $W,$ respectively. For a vector $\bv=(v_1,\ldots, v_r)^\top$ and $g$ a univariate function, we write $g(\bv):=(g(v_1), \ldots, g(v_r))^\top.$ We often apply this to the activation function or the logarithm $g(u)=\log(u).$ Similarly, we define for two vectors of the same length $\bv,\bv',$ $\log(\bv/\bv')=\log(\bv)-\log(\bv').$ For any natural number $\gamma,$ we set $0\log^{\gamma}(0):=0$.
For a real number $x\in\mathbb{R},$ $\floor{x}$ is the largest integer $<x$ and $\ceil{x}$ is the smallest integer $\geq x$.
A $K$-dimensional standard basis vector is a vector of length $K$ that can be written as $(0,\ldots, 0,1,0,\ldots,0)^\top.$ We use $\mathcal{S}^K$ to denote the $(K-1)$-simplex in $\mathbb{R}^K$, that is, $\mathcal{S}^K=\{\mathbf{v}\in\mathbb{R}^K:\sum_{k=1}^Kv_k=1,v_k\geq 0, k=1,\dots,K\}.$ For two probability measures $P$ and $Q$, the Kullback-Leibler divergence $\KL(P,Q)$ is defined as
$\KL(P,Q):=\int\log(dP/dQ)\,dP$ if $P$ is dominated by $Q$ and as $\KL(P,Q):=\infty$ otherwise.

\section{The multiclass classification model}\label{sec.StatModel}
 In multiclass classification with $K\geq 2$ classes and design on $[0,1]^d$, we observe a dataset $\mD_n=\big\{(\bX_i,\bY_i): i=1, \ldots,n\big\}$ of $n$ i.i.d. copies of pairs $(\mathbf{X},\mathbf{Y})$ with design/input vector $\mathbf{X}$ taking values in $[0,1]^d$ and the corresponding response vector $\mathbf{Y}$ being one of the $K$-dimensional standard basis vectors. The response decodes the label of the class: the output $\bY$ is the $k$-th standard basis vector if the label of the $k$-th class is observed. As a special case, for binary classification the output is decoded as $(1,0)^T$ if the first class is observed and as $(0,1)^T$ if the second class is observed. We write $\mathbb{P}$ for the joint distribution of the random vector $(\mathbf{X},\mathbf{Y})$ and $\mathbb{P}_{\mathbf{X}}$ for the marginal distribution of $\mathbf{X}.$ The conditional probability $\mathbb{P}_{\bY|\bX}$ exists since $\bY$ is supported on finitely many points.  

An alternative model is to assume that each of the $K$ classes is observed roughly $n/K$ times. To derive statistical risk bounds, there is hardly any difference and the fact that the i.i.d. model generates with small probability highly unbalanced designs will not change the analysis.

The task is now to estimate/learn from the dataset $\mD_n$ the probability that a new input vector $\bX$ is in class $k$. If $\bY=(Y_1, \ldots, Y_K)^\top,$ the true conditional class probabilities are
$$p^0_{k}(\bx):=\mathbb{P}(Y_k=1|\mathbf{X}=\mathbf{x}), \quad k=1, \ldots, K.$$
For any $\bx$ this gives a probability vector, that is, $\sum_{k=1}^K p^0_{k}(\bx) =1.$ For notational convenience, we also define the vector of conditional class probabilities $\mathbf{p}_{0}(\bx):=(p^0_1(\bx),\cdots,p^0_K(\bx))^\top.$

To learn the conditional class probabilities from data, the commonly employed strategy in deep learning is to minimize the log-likelihood over the free parameters of a deep neural network using (stochastic) gradient descent. The likelihood for the conditional class probability vector $\mathbf{p}(\bx):=(p_1(\bx),\cdots,p_K(\bx))^\top$ is given by
$$\mathcal{L}( \mathbf{p} |\mD_n)=\prod_{i=1}^n\prod_{k=1}^K(p_k(\mathbf{X}_i))^{Y_{ik}},$$ with $Y_{ik}$ the $k$-th entry of $\mathbf{Y}_i.$ The negative log-likelihood or cross-entropy loss is then
\begin{align}\label{eq.CE}
	\bp \mapsto \ell\big(\bp, \mD_n\big):=-\frac{1}{n}\sum_{i=1}^n\sum_{k=1}^KY_{ik}\log(p_{k}(\mathbf{X}_i))
= -\frac{1}{n}\sum_{i=1}^n \bY_i^\top \log\big(\bp(\bX_i)\big),
\end{align}
where the logarithm in the last expression is taken component-wise as explained in the notation section above and $\bY^T \log(\bp(\bX_i))$ is understood as the scalar product of the vectors $\bY$ and $\log(\bp(\bX_i))$. The response vectors $\bY_i$ are standard basis vectors and in particular have nonnegative entries. The cross-entropy loss is thus always nonnegative and consequently defines indeed a proper statistical loss function. The cross-entropy loss is also convex, but not strictly convex and thus also not strongly convex, see \cite{van2014renyiKullbackConvex}, Chapter III-B for a proof. Throughout the article, we consider estimators/learners $\widehat \bp(\mathbf{X})$ with the property that $\widehat\bp(\bx)$ is a probability vector for all $\bx,$ or equivalently, $\widehat \bp(\bx)$ lies in the simplex $\mS^K$ for all $\bx.$ This is in particular true for neural networks with softmax activation function in the output layer. Recall that $\bp_0(\bx)$ is the vector of true class probabilities. If $(\bX,\bY)$ has the same distribution as each of the observations and is independent of the dataset $\mD_n,$ the statistical estimation risk associated with the CE loss is
$$\mathbb{E}_{\mD_n,(\bX,\bY)}\left[\bY^\top \log\Big(\frac{\bp_0(\bX)}{\widehat{\bp}(\bX)}\Big)\right]=\mathbb{E}_{\mD_n,\bX}\left[\bp_0(\bX)^\top \log\Big(\frac{\bp_0(\bX)}{\widehat{\bp}(\bX)}\Big)\right]=\mathbb{E}_{\mD_n,\bX} \big[\KL\big( \bp_0(\bX), \widehat \bp(\bX)\big) \big],$$
where the first equality follows from conditioning on the design vector $\bX$ and $\KL( \bp_0(\bX), \widehat \bp(\bX))$ is understood as the Kullback-Leibler divergence of the discrete distributions with probability mass functions $\bp_0(\bX)|\bX$ and $\widehat\bp(\bX)|(\bX,\mD_n).$

(Stochastic) gradient descent methods aim to minimize the CE loss \eqref{eq.CE} over a function class $\mF$ induced by the method. In the context of neural networks, this class is generated by all network functions with a pre-specified network architecture. In particular, the class is parametrized through the network parameters. The maximum likelihood estimator (MLE) is by definition any global minimizer of \eqref{eq.CE}. For some function classes the MLE can be given explicitly. In the extreme case that $\bx\mapsto \bp(\bx)$ is constraint to constant functions, the problem is equivalent to estimation of the probability vector of a multinomial distribution and the MLE is the average $\widehat \bp^{\MLE} = \tfrac 1n \sum_{i=1}^n\bY_i.$ The other extreme is the case of training error zero. If the observed design vectors are all different, training error zero is achieved whenever there exists $\bp\in \mF$ such that $\bY_i=\bp(\bX_i)$ for all $i=1,\ldots,n.$ This follows from $0\log(0)=1\log(1)=0.$ To achieve training error zero, we therefore need to interpolate all data points. Notice that misclassification error zero does not necessarily require interpolation of the data points. 

Already for small function classes, the MLE has infinite risk if the statistical risk is as defined above. The next lemma makes this precise.

\begin{Lemma}\label{lem.infinite_risk}
Consider binary classification ($K=2$) with uniform design $\bX\sim \Unif([0,1]^d)$ and $\bp_0(\bx):=(1/2,1/2)^\top$ for all $\bx \in [0,1]^d.$ Suppose that the function class $\mF$ contains an element $\bp(\bx)=(p_1(\bx),p_2(\bx))^\top$ such that $p_1(\bx)=0$ for all $\bx\in [0,1/3]^d$ and $p_1(\bx)=1$ for all $\bx\in [2/3,1]^d.$ Then, there exists a MLE $\widehat \bp$ with 
$$\mathbb{E}_{\mD_n,\bX}\left[\bp_0(\bX)^\top \log\Big(\frac{\bp_0(\bX)}{\widehat{\bp}(\bX)}\Big)\right]=\infty.$$
\end{Lemma}

The assumption on the function class $\mF$ in the previous statement is quite weak and is satisfied if $\mF$ contains all piecewise constant conditional class probabilities with at most two pieces or all piecewise linear conditional class probabilities with at most three pieces. A large statistical risk occurs also in the case of zero training error or if the estimator $\widehat \bp$ severely underestimates the true probabilities.

To overcome the shortcomings of the Kullback-Leibler risk, one possibility is to regularize the Kullback-Leibler divergence and to consider for some $B>0$ the truncated Kullback-Leibler risk 
\begin{align*}
	R_B(\bp_0,\widehat \bp):=\E_{\mD_n,\bX} \Big[\KL_B\big( \bp_0(\bX), \widehat \bp(\bX)\big) \Big],
\end{align*}
where
\begin{align*}
	\KL_B\big( \bp_0(\bX), \widehat \bp(\bX)\big):=\bp_0(\bX)^\top\left(B\wedge\log\left(\frac{\bp_0(\bX)}{\widehat \bp(\bX)}\right)\right).
\end{align*}
The loss can be shown to be nonnegative whenever $B\geq 2,$ see Lemma \ref{lem.Hell_KL_bd} below. The threshold $B$ becomes void if the estimator $\widehat \bp$ is constrained to be in $[e^{-B},1]^K.$ If the estimator underestimates one of the true conditional class probabilities by a large factor, the logarithm becomes large and the threshold $B$ kicks in. For $B=\infty,$ we recover the Kullback-Leibler risk.

The idea of truncation is not new. \cite{MR1332570} truncates the log-likelihood ratio to avoid problems with this ratio becoming infinite. Their risk rates, however, are in terms of the Hellinger distance and the truncation does not appear in the statement of their results. For the truncated Kullback-Leibler risk the truncation plays a much more prominent role and appears as a multiplicative factor in the risk bounds. Lemma \ref{lem.Hell_KL_bd} provides insight in this difference: it shows that any upper bound for any $B$-truncated Kullback-Leibler divergence with $B\geq2$ provides an upper bound for the Hellinger distance. 

As we are interested in the multiclass classification problem in the context of neural networks, the function class $\mF$ is not convex. Due to this non-convexity, the training of neural networks does typically not yield a neural network achieving the global minimum. We therefore do not assume that the estimator is the MLE and use a parameter to quantify the difference between the achieved empirical risk and the global minimum: For any estimator $\widehat{\bp}$ taking values in a function class $\mF,$ we denote the difference between $\widehat{\bp}$ and the global minimum of the empirical risk over that entire class by
\begin{equation}\label{Eq: Difference Expected Empirical Risk Estimator vs Minimum}
\Delta_n(\bp_0,\widehat{\bp})
:=\mathbb{E}_{\mD_n}\Big[-\frac{1}{n}\sum_{i=1}^n \bY_i^\top \log(\widehat \bp(\mathbf{X}_i))-\min_{\bp\in\mF} -\frac{1}{n}\sum_{i=1}^n\bY_i^\top \log(\bp(\mathbf{X}_i))\Big].
\end{equation}

\subsection{Deep ReLU networks}

In this work we study deep ReLU networks with softmax output layer. Recall that the rectified linear unit (ReLU) activation function is $\sigma(x):=\max\{x,0\}$. For any vectors $\mathbf{v}=(v_1,\cdots,v_r)^\top, \mathbf{y}=(y_1,\cdots,y_r)^\top\in\mathbb{R}^r,$ write $\sigma_{\mathbf{v}}\mathbf{y}:=(\sigma(y_1-v_1),\dots, \sigma(y_r-v_r))^\top.$ To ensure that the output of the network is a probability vector over the $K$ classes, it is standard to apply the softmax function $$\bPhi= \bigg(\frac{e^{x_1}}{\sum_{j=1}^Ke^{x_j}}, \dots, \frac{e^{x_K}}{\sum_{j=1}^Ke^{x_j}}\bigg) :\mathbb{R}^K\rightarrow \mathcal{S}^K$$ in the last layer. We use $L$ to denote the number of hidden layers or depth of the neural network, and $\bfm=(m_0,\cdots,m_{L+1})\in\mathbb{N}^{L+2}$ to denote the widths, that is, the number of nodes in each layer of the network. A (ReLU) network architecture with output function $\bpsi:\mathbb{R}^{m_{L+1}}\rightarrow \mathbb{R}^{m_{L+1}}$ is a pair $(L,\bfm)_{\bpsi}$ and
a network with network architecture $(L,\bfm)_{\bpsi}$ is any function of the form
\begin{equation}\label{Eq: ReLU architecture function}
\bf:\mathbb{R}^{m_0}\rightarrow\mathbb{R}^{m_{L+1}},\ \ \mathbf{x}\mapsto \bf(\bx)=\bpsi W_{L}\sigma_{\mathbf{v}_L}W_{L-1}\sigma_{\mathbf{v}_{L-1}}\cdots W_{1}\sigma_{\mathbf{v}_{1}}W_0\mathbf{x},
\end{equation}
where $W_j$ is a $m_{j}\times m_{j+1}$ weight matrix and $\mathbf{v}_j\in\mathbb{R}^{m_j}$ is a shift vector. Throughout this paper we use the convention that $\bv_0:=(0,\dots,0)^\top \in \mathbb{R}^{m_0}.$

First we define neural network classes with the additional property that all network parameters are bounded in absolute value by one via
\begin{equation*}
\mF_{\bpsi}(L,\mathbf{m}):=\left\{\bf \text{ is of the form of \eqref{Eq: ReLU architecture function}}: \max_{j\in\{0,\cdots,L\}}(\|W_j\|_{\infty}\vee \|\mathbf{v}_j\|_{\infty})\leq 1 \right\},
\end{equation*}
with the maximum entry norm $\|\cdot \|_\infty$ as defined in the notation section above. As in previous work, we study estimation over $s$-sparse ReLU networks. Those are function classes of the form
\begin{equation*}
\begin{aligned}
\mF_{\bpsi}(L,\bfm,s)&:= \bigg\{\bf\in\mF(L,\bfm):\sum_{j=0}^{L}\|W_j\|_0+\|\mathbf{v}_j\|_0\leq s \bigg\},
\end{aligned}
\end{equation*}
where the counting norm $\|\cdot\|_{0}$ denotes the number of nonzero vector/matrix entries.

All neural network classes in this work have either softmax output activation $\bpsi=\bPhi$ or identity output activation $\bpsi=\id.$

\section{Main Results}\label{sec.results}

Interesting phenomena occur if the conditional class probabilities are close to zero or one. We now introduce a notion measuring the size of the set on which the conditional class probabilities are small. The index $\alpha$ will later appear in the convergence rate.

\begin{Def}(Small Value Bound)
Let $\alpha\geq 0$ and $\mH$ be a function class. We say that $\mH$ is $\alpha$-small value bounded (or $\alpha$-SVB) if there exists a constant $C>0,$ such that for all $\bp=(p_1,\dots,p_K)\in\mH$ it holds that
\begin{equation*}
    \mathbb{P}_{\bX}(p_k(\bX)\leq t)\leq Ct^{\alpha}, \quad \text{for all} \ t\in(0,1] \ \text{and all} \ k\in\{1,\dots,K\}. 
\end{equation*}
\end{Def}
The condition always holds for $\alpha=0$ and $C=1$. If $\mathbb{P}_{\bX}(p_k(\bX)=0)>0,$ the condition does not hold for $\alpha>0.$
If all functions in a class are lower bounded by a constant $B_0$, the class is $\alpha$-SVB for any $\alpha$ with constant $C=B_0^{-\alpha}$. More generally, the index $\alpha$ is completely determined by the behaviour near zero: If for some function class there exists some $0<\tau\ll 1,$ so that the bound holds for $\alpha$ and for all $t\in(0,\tau],$ then replacing $C$ by $C'=\max\{C,\tau^{-\alpha}\}$ guarantees that $C'\tau^{\alpha}\geq 1$, which in turn implies that the function class is $\alpha$-SVB. Moreover, if a function class is $\alpha$-SVB, then it is also $\alpha^*$-SVB for all $\alpha^*\leq \alpha$. This follows immediately by noticing that $t^{\alpha^*}\geq t^{\alpha}$ for all $t\in(0,1].$ 
Increasing the index makes the small value bound condition thus more restrictive.

To show that the definition of the small value bound makes sense, we have to check that for any $\alpha>0,$ there exist conditional class probabilities that are $\alpha$-SVB for that $\alpha,$ but are not $\alpha^*$-SVB for any larger $\alpha^*>\alpha.$ To see this, consider the case that $X$ is uniformly distributed on $[0,1],$ and that there are three classes $K=3.$ For given $\alpha>0,$ define the function $\bp_{\alpha}:[0,1]\rightarrow \mathcal{S}^3$ as $p_1(x)=\min\{x^{1/\alpha},1/3\}$, $p_2(x)=1/3$ and $p_3(x)=1-p_1(x)-p_2(x)=2/3-\min\{x^{1/\alpha},1/3\}$. Since $p_2(x),p_3(x)\geq 1/3,$ we have for $k=2,3$ that $\mathbb{P}_{X}(p_k(X)\leq t)\leq (3t)^{\alpha}.$
When $k=1$, it holds for $t\leq 1/3$ that
$\mathbb{P}_{X}(p_1(X)\leq t)=\mathbb{P}_{X}(X^{1/\alpha}\leq t)=\mathbb{P}_{X}(X\leq t^{\alpha})=t^{\alpha}.$ Hence $\mathbb{P}_{X}(p_k(X)\leq t)\leq (3t)^{\alpha}$ for $k=1,2,3,$ so  $\bp_{\alpha}$ is $\alpha$-SVB with constant $3^{\alpha}.$ Now we show that this function is not $\alpha^*$-SVB for any $\alpha^*>\alpha$. Let $\alpha^*>\alpha,$ then for every constant $C>0,$ there exists a $\tau_C\in(0,1/3)$ such that $C(\tau_C)^{\alpha^*}<(\tau_C)^{\alpha}=\mathbb{P}_{X}(p_1(X)\leq \tau_C).$ Since $C$ is arbitrary, $\bp_{\alpha}$ is not $\alpha^*$-SVB.

The following theorem shows the influence of the index $\alpha$ in the small value bound on the approximation rates.

\begin{Stelling}\label{S:ApproxUnderCondA}
If the function class is $\alpha$-SVB with constant $C,$ then, for any approximating function $\bp=(p_1,\dots,p_k):[0,1]^d\rightarrow \mS^K$ satisfying $\|\bp-\bp_0\|_{\infty}\leq C_1/M,$ and $\min_k \inf_{\bx\in [0,1]^d} p_k(\bx)\geq 1/M,$ for some constant $C_1,$ it holds that
\begin{equation*}
    \mathbb{E}_{\bX}\left[(\bp_0(\bX))^\top\log\left(\frac{\bp_0(\bX)}{\bp(\bX)}\right)\right]\leq 
    CK\frac{(C_1+1)^{2+(\alpha\wedge 1)}}{M^{1+(\alpha\wedge 1)}}
    \Big(
    1+\frac{\1_{\{\alpha<1\}}}{1-\alpha}
    +\log(M)\Big).
\end{equation*}
\end{Stelling}

The proof for this result bounds the Kullback-Leibler divergence by the $\chi^2$-divergence and then distinguishes the cases where the conditional class probabilities are smaller and larger than $1/M.$ Both terms can be controlled via the $\alpha$-SVB condition. The convergence rate becomes faster in $M$ up to $\alpha=1$ and is $\log(M)/M^2$ for all $\alpha\geq 1.$

The small value bound has a similar flavor as Tsybakov's margin condition, which can be stated as $\mathbb{P}_{\bX}(0<|p_0(\bX)-1/2|\leq t)\leq Ct^{\gamma}$ for binary classification \cite{PlugInRatesMarginCondition}. The margin condition provides a control on the number of data points that are close to the decision boundary $\{\bx:p_0(\bx)=1/2\}$ and that are therefore hard to classify correctly. Differently speaking, the problem becomes easier if the conditional class probabilities are either close to zero or one. This is in contrast with the small value bound, which will lead to faster convergence rates when the true conditional class probabilities are mostly away from zero. This difference is due to the loss: the $0$-$1$ loss only cares about predicting the class membership, while the CE loss measures how well the conditional class probabilities are estimated and puts additional emphasis on small conditional class probabilities by considering the ratio between prediction and truth.

To obtain estimation rates, we further assume that the underlying true conditional class probability function $\bp_0$ belongs to the class of H\"older-smooth functions.
For $\beta>0$ and $D\subset \mathbb{R}^m$, the ball of $\beta$-H\"{o}lder functions with radius $Q$ is defined as
\begin{equation*}
    C^{\beta}(D,Q)=\left\{f:D\rightarrow \mathbb{R}: \sum_{\bgamma:\|\bgamma\|_1<\beta}\|\partial^{\bgamma}f\|_{\infty}+\sum_{\bgamma:\|\bgamma\|_1=\floor{\beta}}\sup_{\bx,\by\in D, \bx\neq\by}\frac{|\partial^{\bgamma}f(\bx)-\partial^{\bgamma}f(\by)|}{\|\bx-\by\|^{\beta-\floor{\beta}}_{\infty}}\leq Q \right\},
\end{equation*}
where $\partial^{\bgamma}=\partial^{\gamma_1}\dots \partial^{\gamma_m},$ with $\bgamma=(\gamma_1,\dots, \gamma_m)\in\mathbb{N}^m.$
The function class $\mG(\beta,Q)$ of $\beta$-smooth conditional class probabilities is then defined as
\begin{equation*}
    \mG(\beta,Q)=\left\{\bp=(p_1,\cdots,p_K)^{\top}:[0,1]^d\rightarrow \mathcal{S}^K: p_k\in C^{\beta}([0,1]^d,Q), k=1,\dots, K\right\}.
\end{equation*} 
If $Q<1/K,$ then, $\|\bp\|_{\infty}\leq Q$ implies $\sum_{k=1}^K p_k\leq KQ<1,$ so we need H\"older radius $Q\geq 1/K$ for this class to be non-empty.
Combining the smoothness and the small value bound, we write $\mG_{\alpha}(\beta,Q)=\mG_{\alpha}(\beta,Q,C)$ for all functions in $\mG(\beta,Q)$ that satisfy the  $\alpha$-SVB condition with constant C. For large enough radius $Q$ and constant $C,$ the class $\mG_{\alpha}(\beta,Q)$ is non-empty. For example, the constant function $\bp=(1/K,\dots,1/K)$ is in $\mG_{\alpha}(\beta,Q)$ for any $\beta>0$ and $\alpha>0$ when $Q\geq 1/K$ and $C\geq K^{\alpha}$.

For $0\leq \alpha \leq 1$ the index from the SVB condition and $\beta$ the smoothness index, we introduce the rate
\begin{equation*}
    \phi_n= K^{\frac{(1+\alpha)\beta+(3+\alpha)d}{(1+\alpha)\beta+d}} n^{-\frac{(1+\alpha)\beta}{(1+\alpha)\beta+d}}.
\end{equation*}

\begin{Stelling}[Main Risk Bound]\label{S: Main Risk Bound}
Consider the multiclass classification model with $\bp_0\in\mG_{\alpha}(\beta,Q),$ $0\leq \alpha \leq 1,$ and $n>1.$ Let $\widehat{\bp}$ be an estimator taking values in the network class $\mF_{\bPhi}(L,\bfm,s)$ satisfying
\begin{compactitem}
    \item [(i)] $A(d,\beta)\log_2(n)\leq L\lesssim n\phi_n,$
    \item[(ii)] $\min_{i=1,\cdots,L} m_i \gtrsim n\phi_n,$
    \item[(iii)] $s\asymp n\phi_n\log(n)$
\end{compactitem}
for a suitable constant $A(d,\beta).$ If $n$ is sufficiently large, then, there exist constants $C',C''$ only depending on $\alpha, C, \beta, d$, such that whenever $\Delta_n(\widehat{\bp},\bp_0)\leq C'' B\phi_nL\log^2(n)$ then
\begin{equation*}
    R_B(\bp_0,\widehat{\bp})\leq C'B \phi_nL\log^2(n).
\end{equation*}
\end{Stelling}

An explicit expression for the constant $A(d,\beta)$ can be derived from the proof. The risk bound depends linearly on $B$. Choosing, for instance, $B=O(\log(n))$ leads only to an additional logarithmic factor in the convergence rate. The risk bound grows with $K^{\frac{(1+\alpha)\beta+(3+\alpha)d}{(1+\alpha)\beta+d}}$ in the number of classes. Thus for large $\beta$, we obtain a near linear dependence on $K.$ The worst behavior occurs for $\alpha=1$ and $d$ large. Then the dependence on the number of classes is essentially of the order $K^4.$

When the estimator $\widehat{\bp}$ is guaranteed to have output in $[e^{-B},1]^K,$ the truncation parameter $B$ in the risk has no effect. The proof of the approximation properties is done by the construction of a softmax-network $\widehat{\bg}$ with the property that $\widehat{\bg}(\bx)\gtrsim K^{\frac{-(2+\alpha)\beta}{(1+\alpha)\beta+d}} n^{-\frac{\beta}{(1+\alpha)\beta+d}},$ for all $\bx \in [0,1]^d.$ This means that we can pick $B\asymp \log(n)$ such that $\widehat{\bg}(\bx)\geq e^{-B}$ and restrict the class $\mF_{\bPhi}(L,\bfm,s)$ to networks that are guaranteed to have output in $[e^{-B},1]^K$. The proof of Theorem \ref{S: Main Risk Bound} can be extended for this setting and implies a risk bound for the Kullback-Leibler risk of the form
$$\E_{\mD_n,\bX} \big[\KL\big( \bp_0(\bX), \widehat \bp(\bX)\big) \big]\leq C''' \phi_nL\log^3(n),$$ for some constant $C'''.$ Thus Theorem \ref{S: Main Risk Bound} provides us with rates for the Kullback-Leibler risk when the networks outputs are guaranteed to be sufficiently large, while still providing a bound for the truncated Kullback-Leibler risk when no such guarantee can be given.

When the input dimension $d$ is large, the obtained convergence rates become slow. A possibility to circumvent this curse of dimensionality is to assume additional structure on $\bp_0.$ For nonparametric regression,  \cite{horowitz2007, kohler2017, bauer2019, NonParametricRegressionReLU, kohler2020rate2} show that under a composition assumption on the regression function, neural networks can exploit this structure to obtain fast convergence rates that are unaffected by the curse of dimensionality. It is possible to incorporate such compositions assumptions here as well and to obtain similar convergence rates.

\subsection{Relationship with Hellinger distance}
The multiclass classification problem can be written as statistical model $(Q_{\bp},\bp \in \mF),$ where $\mF$ is the parameter space, $\bp$ is the unknown vector of conditional class probabilities and $Q_{\bp}$ denotes the data distribution if the data are generated from the conditional class probabilities $\bp.$ The squared Hellinger distance $H(P,Q)^2=\tfrac 12 \int (\sqrt{dP}-\sqrt{dQ})^2,$ with $P$ and $Q$ probability measures on the same probability space, induces in a natural way a loss function on such a statistical model by associating to the two parameters $\bp$ and $\bp'$ the loss $H(Q_{\bp},Q_{\bp'}).$ The Hellinger loss function has been widely studied in the context of nonparametric variations of the maximum likelihood principle, mainly for the related nonparametric density estimation problem, \cite{MR1332570, vdGeer2000, WeakConvergenceEmpProcesses}. The log-likelihood is closely related to the Kullback-Leibler divergence, which in turn is related to the Hellinger distance by the inequality $H(P,Q)^2\leq \KL(P,Q),$ see for example \cite{tsybakov2008introduction}. The Kullback-Leibler divergence cannot be upper bounded by the squared Hellinger distance in general, although there exists conditions under which such a bound can be established, see for example Theorem 5 of \cite{MR1332570} and Lemma \ref{lem.Hell_KL_bd} below.

In density estimation, the nonparametric MLE achieves in some regimes optimal rates with respect to the Hellinger distance for convex estimator classes or if the densities (or  sieve estimators) are uniformly bounded away from zero, see \cite{MR1105838, MR1332570} and Chapters 7 and 10 in \cite{vdGeer2000}. Neural network function classes are not convex and, as argued before, there are many applications in the deep learning literature, where the conditional class probabilities are very small or even zero. Thus, these general results are not applicable in our setting.

On the contrary, the convergence rates established above for the truncated Kullback-Leibler divergence imply convergence with respect to the Hellinger loss. This relationship is made precise in the next result. 

\begin{Lemma} \label{lem.Hell_KL_bd} Let $P$ and $Q$ be two probability measures defined on the same measurable space. For any $B\geq 2$,
\begin{align*}
	H^2(P,Q)\leq \frac 12 \KL_2(P,Q)\leq \frac 12 \KL_{B}(P,Q) \leq 2e^{B/2} H^2(P,Q).
\end{align*}
\end{Lemma}
For the proof see Appendix \ref{appendix.proofs}. The upper bound on the truncated Kullback-Leibler divergence is related to the inequalities that bound the Kullback-Leibler divergence by the squared Hellinger distance under the assumption of a bounded likelihood ratio, such as (7.6) in \cite{MR1653272} or Lemma 4 in \cite{MR1604481}.

Combining the previous lemma and Theorem \ref{S: Main Risk Bound} with $B=2$ gives
\begin{align}
    \E_{\mD_n}\Big[\int_{[0,1]^d} \sum_{j=1}^K \Big(\sqrt{p_j^0(\bx)}-\sqrt{\widehat p_j(\bx)}\Big)^2 \, d\mathbb{P}_{\mathbf{X}}(\bx)\Big]
    \leq 2C'\phi_nL\log^2(n),
    \label{eq.sq_Hell_bound}
\end{align}
whenever $\Delta_n(\widehat{\bp},\bp_0)\leq C'' B\phi_nL\log^2(n).$ 

We can also use the relation with the Hellinger distance to show that for $\alpha=1,$ we obtain a near minimax optimal convergence rate. Indeed $n^{-\frac{2\beta}{2\beta+d}}$ is the optimal rate for the squared Hellinger distance. For references see for instance Example 7.4.1 of \cite{vdGeer2000} for univariate densities bounded away from zero; the entropy bounds in Theorem 2.7.1. together with Proposition 1 of \cite{YangBarronMiniMax} for densities bounded away from zero; or the entropy bounds in Theorem 2.7.1. and Equation (3.4.5) of \cite{WeakConvergenceEmpProcesses} together with Chapter 2.3. of \cite{YangBarronMiniMax} for densities $p$ for which $\int\frac{1}{p}$ is bounded. Since the squared Hellinger distance can be upper bounded by the Kullback-Leibler divergence, the rate $n^{-\frac{2\beta}{2\beta+d}}$ is also a lower bound for the Kullback-Leibler risk. Since this rate is achieved for $\alpha=1$, it is clear that no further gain in the convergence rate can be expected for $\alpha>1.$ For $\alpha\geq 1,$ the rate of convergence is up to $\log(n)$-factors the same as in Theorem 5 of \cite{MR1332570} and also the conditions are comparable.

It is instructive to relate the global convergence rates to pointwise convergence. Recall that for real numbers $a,b,$ we have  $(\sqrt{a}-\sqrt{b})^2=(a-b)^2/(\sqrt{a}+\sqrt{b})^2.$ If $\mathbb{P}_{\mathbf{X}}$ has a Lebesgue density that is bounded on $[0,1]^d$ from below and above and if we choose $L$ of the order $O(\log n),$ \eqref{eq.sq_Hell_bound} indicates that on a large subset of $[0,1]^d,$ we can expect a pointwise distance
\begin{align*}
    \Big|p_j^0(\bx)-\widehat p_j(\bx)\Big|
    \lesssim \Big|\sqrt{p_j^0(\bx)}+\sqrt{\widehat p_j(\bx)}\Big| 
    K^{\frac{(1+\alpha/2)d}{(1+\alpha)\beta+d}} n^{-\frac{(1+\alpha)\beta}{2(1+\alpha)\beta+2d}}\log^{3/2}(n).
\end{align*}
The pointwise convergence rate gets therefore faster if the conditional class probabilities are small. In the most extreme case, $p_j^0(\bx)=0,$ the previous bound becomes
\begin{align*}
    \Big|p_j^0(\bx)-\widehat p_j(\bx)\Big|
    \lesssim
    K^{\frac{(2+\alpha)d}{(1+\alpha)\beta+d}} n^{-\frac{(1+\alpha)\beta}{(1+\alpha)\beta+d}}\log^{3}(n).
\end{align*}
Since $n^{-(1+\alpha)\beta/((1+\alpha)\beta+d)}\ll n^{-\beta/(2\beta+d)},$ this rate can be much faster than the classical nonparametric rate for pointwise estimation $n^{-\beta/(2\beta+d)}$. The gain gets accentuated as the index $\alpha$ increases. A large index $\alpha$ in the SVB bound can be chosen if the conditional class probabilities are rarely small or zero. Hence there is a trade-off and the regions on which a faster rate can be obtained are thus smaller.

\subsection{Oracle Inequality}
The risk bound of Theorem \ref{S: Main Risk Bound} relies on an oracle-type inequality. Before we can state this inequality we first need some definitions. Given a function class of conditional class probabilities $\mF,$ we denote by $\log(\mF)$ the function class containing all functions that can be obtained by applying the logarithm coefficient-wise to functions from $\mF$, that is,
\begin{equation*}
\log(\mF)=\big\{\mathbf{g}=\log(\bf): \bf \in \mF\big\}.
\end{equation*} 
Next we define a family of pseudometrics. Recall that a pseudometric is a metric without the condition that $d(f,g)=0$ implies $f=g$.
For a real number $\tau$ and $\bf,\bg:\mathcal{D}\rightarrow \mathbb{R}^K$, set
\begin{align*}
    d_{\tau}(\bf,\bg):=\sup_{\bx\in\mathcal{D}}\,\,\max_{k=1,\cdots,K}|(\tau\vee f_k(\bx))-(\tau\vee g_k(\bx))|.
\end{align*}
Lemma \ref{P: truncation log pseudometric} in the appendix verifies that this indeed defines a pseudometric. For $\tau=-\infty,$ $d_{\tau}(\bf,\bg)$ coincides with the $L^\infty$-norm as defined in the notation section. 

Denote by $\mathcal{N}(\delta,\mF,d(\cdot,\cdot))$ the $\delta$ interior covering number of a function class $\mF$ with respect to a (pseudo)metric $d(.,.)$. For interior coverings, the centers of the balls of any cover are required to be inside the function class $\mF$. Triangle inequality shows that any (exterior) $\delta$-cover can be used to construct an interior cover with the same number of balls, but with radius $2\delta$ instead of $\delta$.

\begin{Stelling}[Oracle Inequality]\label{S: Main Oracle Inequality}
Let $\mF$ be a class of conditional class probabilities and $\widehat{\bp}$ be any estimator taking values in $\mF$. If $B\geq 2$ and $\mathcal{N}_n=\mathcal{N}(\delta,\log(\mF),d_{\tau}(\cdot,\cdot))\geq 3$ for $\tau=\log(C_ne^{-B}/n)$, then
\begin{equation*}
\begin{aligned}
R_B(\bp_0,\widehat \bp)&\leq (1+\epsilon)\left(\inf_{\bp\in\mF}R(\bp_0,\bp)+\Delta_n(\bp_0,\widehat{\bp})+3\delta\right)\\
&+\frac{(1+\epsilon)^2}{\epsilon}\cdot\frac{68B\log(\mathcal{N}_n)+272B+(3/2)C_nK\big(\log\big(\frac{n}{C_n}\big)+B\big)}{n},
\end{aligned}
\end{equation*}
for all $\delta,\epsilon\in(0,1]$, $0<C_n\leq ne^{-1}$ and $\Delta_n(\bp_0,\widehat{\bp})$ as defined in \eqref{Eq: Difference Expected Empirical Risk Estimator vs Minimum}.
\end{Stelling}
The proof of this oracle inequality is a non-trivial variation of the proof for the oracle inequality in the regression model \cite{NonParametricRegressionReLU}. The statement seems to suggest to pick a small $C_n.$ Then, however, also $\tau$ will be small, and $d_\tau$ becomes a stronger metric possibly leading to an increase of the covering number $\mathcal{N}_n.$

We can also replace the covering number of $\log(\mF)$ by the covering number of $\mF$ in the oracle inequality:
\begin{Corollary}
Denote $\widetilde{\mathcal{N}}_n:=\mathcal{N}(\delta C_ne^{-B}/n,\mF,d_{\tau}(\cdot,\cdot))$, with $\tau=C_ne^{-B}/n$.
Under the conditions of Theorem \ref{S: Main Oracle Inequality}, it holds that
\begin{equation*}
\begin{aligned}
R_B(\bp_0,\widehat \bp)&\leq (1+\epsilon)\left(\inf_{\bp\in\mF}R(\bp_0,\bp)+\Delta_n(\bp_0,\widehat{\bp})+3\delta\right)\\
&+\frac{(1+\epsilon)^2}{\epsilon}\cdot\frac{68B\log(\widetilde{\mathcal{N}}_n)+272B+(3/2)C_nK(\log(n/C_n)+B)}{n},
\end{aligned}
\end{equation*}
for all $\delta,\epsilon\in(0,1]$, $0<C_n\leq ne^{-1}$ and $\Delta_n(\bp_0,\widehat{\bp})$ as defined in \eqref{Eq: Difference Expected Empirical Risk Estimator vs Minimum}.
\end{Corollary}

Let us briefly discuss some ideas underlying the proof of the oracle inequality. For simplicity, assume that $\widehat \bp$ is the MLE over a class $\mF$ and that $\bp_0\in \mF.$ By the definition of the MLE $\widehat \bp$, we have that $-\tfrac 1n\sum_{i=1}^n \bY_i^\top \log(\widehat \bp(\bX_i))\leq -\tfrac 1n \sum_{i=1}^n \bY_i^\top \log(\bp_0(\bX_i)).$ Taking expectation on both sides, one can then show that for any $B\geq 0,$
\begin{align*}
	\E_{\mD_n}\Big[ \frac 1n \sum_{i=1}^n \bp_0(\bX_i)^\top \Big(B \wedge  \log\Big(\frac{\bp_0(\bX_i)}{\widehat \bp(\bX_i)}\Big)\Big) \Big]
	\leq  \E_{\mD_n}\Big[ \frac 1n \sum_{i=1}^n \big(\bp_0(\bX_i)-\bY_i\big)^\top \Big( B \wedge \log\Big(\frac{\bp_0(\bX_i)}{\widehat \bp(\bX_i)}\Big)\Big)\Big].
\end{align*}
Using standard empirical process arguments, the right hand side can be roughly upper bounded by $\E_{\mD_n} [\max_j \frac 1n \sum_{i=1}^n (\bp_0(\bX_i)-\bY_i)^\top (B \wedge \log(\bp_0(\bX_i)/ \bp_j(\bX_i)))],$ where the maximum is over all centers of an $\eps$-covering of $\mF$ for a sufficiently small $\eps$. Since $\E_{\mD_n}[\bY_i | \bX_i]=\bp_0(\bX_i),$ this is the maximum over a centered process. Using empirical process theory a second time, the left hand side of the previous display can be shown to converge to the statistical risk $R_B(\bp_0, \widehat \bp)=\E_{\mD_n,\bX} [\KL_B( \bp_0(\bX), \widehat \bp(\bX))].$

To apply Bernstein's inequality we need to bound the moments of the random variables in the empirical process. For that we have derived the following inequality that relates the $m$-th moment to the truncated Kullback-Leibler divergence and also shows the effect of the truncation level $B.$

\begin{Lemma}\label{L: Inequality to relate to the risk}
If $B>1$ and $m=2,3,\dots$, then, for any two probability vectors $(p_1,\dots,p_K)$ and $(q_1,\dots,q_K),$ we have
\begin{equation*}
\sum_{k=1}^Kp_k\left|B\wedge\log\left(\frac{p_k}{q_k}\right)\right|^m 
\leq 
\max\left\{m!,\frac{B^m}{B-1}\right\} \sum_{k=1}^Kp_k\left(B\wedge\log\left(\frac{p_k}{q_k}\right)\right).
\end{equation*}
\end{Lemma}

In order to use the oracle inequality for deep ReLU networks with softmax activation in the output layer, we now state a bound on the covering number of these classes.
The bound and its proof are a slight modification of Lemma 5 in \cite{NonParametricRegressionReLU}.

\begin{Lemma}\label{L: Covering Number bound}
If $V:=\prod_{\ell=0}^{L+1}(m_{\ell}+1)$, then for every $\delta>0$,
\begin{equation*}
\mathcal{N}\left(\delta,\log(\mF_{\bPhi}(L,\bfm,s)),\|\cdot\|_{\infty}\right)\leq \left(4\delta^{-1}K(L+1)V^2\right)^{s+1},
\end{equation*}
and
\begin{equation*}
    \log\mathcal{N}\left(\delta,\log(\mF_{\bPhi}(L,\bfm,s)),\|\cdot\|_{\infty}\right)\leq (s+1)\log(2^{2L+6}\delta^{-1}(L+1)K^3d^2s^L).
\end{equation*}
\end{Lemma}

The second bound follows from the first by removing inactive nodes, Proposition \ref{P: Inactive Node Removal}, and taking the logarithm. The full proof can be found in  Appendix \ref{appendix.proofs}.

The proof of the main risk bound in Theorem \ref{S: Main Risk Bound} is based on the oracle inequality derived above. To bound the individual error terms, we apply the approximation theory developed in Theorem \ref{S:ApproxUnderCondA} and Lemma \ref{L:ApproximationSoftmaxNetwork} as well as the previous bound on the metric entropy. This shows that for any $M>1,$ the truncated Kullback-Leibler risk for a network class with depth $L,$ width $\gtrsim KM^{d/\beta}$ and sparsity $s\lesssim KM^{d/\beta}$ can be bounded by
\begin{align*}
    R_B(\bp_0,\widehat \bp)
    &\lesssim 
    K^{3+\alpha}\frac{\log(M)}{M^{1+\alpha}}+KM^{d/\beta}L \frac{\log^2(n)}n+\Delta_n(\widehat{\bp},\bp_0).
\end{align*}
Balancing the terms $K^{3+\alpha}/M^{1+\alpha}$ and $KM^{d/\beta}$ leads to $M \asymp K^{\frac{(2+\alpha)\beta}{(1+\alpha)\beta+d}}n^{\frac{\beta}{(1+\alpha)\beta+d}}$ and for small $\Delta_n(\widehat{\bp},\bp_0),$ we get the rate $R_B(\bp_0,\widehat \bp)\lesssim K^{\frac{(1+\alpha)\beta+(3+\alpha)d}{(1+\alpha)\beta+d}} n^{-\frac{(1+\alpha)\beta}{(1+\alpha)\beta+d}} L\log^2(n)$ in Theorem \ref{S: Main Risk Bound}.

\section{Proofs}

\begin{proof}[Proof of Lemma \ref{lem.infinite_risk}]
Consider the event $\mA_n:=\{(\bX_i, \bY_i)\in ([0,1/3]^d\times (1,0)^\top) \cup ([2/3,1]^d \times (0,1)^\top), \ \text{for all}  \ i=1,\ldots, n \}.$ Recall that $0\log(0)=0.$ On the event $\mA_n$, for any $\bp(\bx)=(p_1(\bx),p_2(\bx))^\top$ such that $p_1(\bx)=0$ for all $\bx\in [0,1/3]^d$ and $p_1(\bx)=1$ for all $\bx\in [2/3,1]^d,$ we have that $\ell(\bp,\mD_n)=0.$ Since the CE loss is nonnegative, any such $\bp$ in the class $\mF$ is a MLE on this event. Since $\mathbb{P}(\mA_n)>0$, it follows that
\begin{align*}
    \mathbb{E}_{\mD_n,\bX}\left[\bp_0(\bX)^\top \log\Big(\frac{\bp_0(\bX)}{\widehat{\bp}(\bX)}\Big)\right]
    &\geq \mathbb{E}_{\mD_n}\left[\mathbf{1}(\mA_n) \int_{[0,1]^d} 
    \bp_0(\bu)^\top \log\Big(\frac{\bp_0(\bu)}{\widehat{\bp}(\bu)}\Big) \, d\bu 
     \right] \\
     &= \infty \cdot \mathbb{P}(\mA_n) \\
     &= \infty.
\end{align*}
\end{proof}

\subsection{Approximation related results}
This section is devoted to the proof of Theorem \ref{S: Main Risk Bound}. First we construct a neural network that approximates $\bp_0$ in terms of the $L_{\infty}$-norm and is bounded away from zero. Afterwards we prove Theorem \ref{S:ApproxUnderCondA} relating the previously derived approximation theory to a bound on the approximation error in terms of the expected Kullback-Leibler divergence. We finish the proof combining this network with the new oracle inequality (Theorem \ref{S: Main Oracle Inequality}) and an entropy bound for classes of neural networks with a softmax function in the output layer, Lemma \ref{L: Covering Number bound}. Recall that $\mF_{\id}(L,\bfm,s)$ denotes the neural network class with $L$ hidden layers, width vector $\bfm$, network sparsity $s$ and identity activation function in the output layer.

\begin{Stelling}\label{T:MSTLogNetworkResult}
For all $M\geq 2$ and $\beta>0$ there exists a neural network $G\in\mF_{\id}(L,\bfm,s)$, with depth $L=\floor{40(\beta+2)^2\log_2(M)},$ width $\bfm=(1,\allowbreak \floor{48\ceil{\beta}^32^{\beta}M^{1/\beta}},\cdots,\floor{48\ceil{\beta}^32^{\beta}M^{1/\beta}},1)$ and sparsity $s\leq 4284(\beta+2)^52^{\beta}M^{1/\beta}\log_2(M),$ such that for all $x\in [0,1],$
\begin{equation*}
\big|e^{G(x)}-x\big|\leq \frac{4}{M} \ \ \text{and} \ \ 
    G(x)\geq \log\Big(\frac 4M\Big).
\end{equation*}
\end{Stelling}
The proof of this theorem can be found in Appendix \ref{A:Proof of LogNetwork}. To approximate H\"{o}lder functions we use Theorem 5 from \cite{NonParametricRegressionReLU} with $m$ equal to $\ceil{\log_2(M))(d/\beta+1)}$. We state here a variation of that theorem in our notation using weaker upper bounds to simplify the expressions for the network size. These upper bounds can be deduced directly from the depth-synchronization and network enlarging properties of neural networks stated in Section \ref{sec.embed_props}. Set
$$C_{Q,\beta,d}:=(2Q+1)(1+d^2+\beta^2)6^d+Q3^{\beta}.$$

\begin{Stelling}\label{T:ReluNWApproxResult}
For every function $\bf\in \mathcal{G}(\beta,Q)$ and every $M>(\beta+1)^{\beta}\vee(Q+1)^{\beta/d}e^{\beta},$ there exist neural networks $H_k\in\mF_{\id}(L,\bfm,s)$ with
 $L=3\ceil{\log_2(M)(d/\beta+1)}(1+\ceil{\log_2(d\vee\beta)}),$  $\bfm=(d,6(d+\ceil{\beta})\floor{M^{d/\beta}},\cdots,6(d+\ceil{\beta})\floor{M^{d/\beta}},1)$, and $s\leq 423(d+\beta+1)^{3+d}M^{d/\beta}\log_2(M))(d/\beta+1),$ such that 
\begin{equation*}
\big\|H_k-f_k^0\big\|_{\infty}\leq \frac{C_{Q,\beta,d}}{M}, \ \ \forall k\in\{1,\cdots,K\}.
\end{equation*} 
\end{Stelling}
Here the $M$ is chosen such that $M^{d/\beta}\asymp N,$ where $N$ is as defined in Theorem 5 of \cite{NonParametricRegressionReLU}.

Without loss of generality we can assume that the output of the $H_k$ networks lies in $[0,1]$. Indeed if this would not be the case, then the projection-layer that we use later on in our proof will guarantee that it is in this interval. This will not increase the error since the functions $f_k^0$ only take values in $[0,1]$.

To obtain a neural network with softmax output, the next lemma combines the neural network constructions from the previous two theorems and replaces the output with a softmax function.

\begin{Lemma}\label{L:ApproximationSoftmaxNetwork}
For every function $\bf\in \mathcal{G}(\beta,Q)$ and every $M> K(4+C_{Q,\beta,d})\vee  (\beta+1)^{\beta}\vee(Q+1)^{\beta/d}e^{\beta}$, there exists a neural network $\widetilde{\bq}\in \mF_{\bPhi}(L,\bfm,s)$, with $L=3\ceil{\log_2(M)(d/\beta+1)}(1+\ceil{\log_2(d+\beta)})+\floor{40(\beta+2)^2\log_2(M)}+2,$ $$\bfm=\Big(d,\floor{48K(d+\ceil{\beta}^3)2^{\beta}M^{d/\beta}},\cdots, \floor{48K(d+\ceil{\beta}^3)2^{\beta}M^{d/\beta}},K\Big),$$ and $s\leq 4707K(d+\beta+1)^{4+d}2^{\beta}M^{d/\beta}\log_2(M))(d/\beta+1),$
 such that,
 \begin{equation*}
     \|\widetilde{\mathbf{q}}_k-\bp_0\|_{\infty}\leq \frac{2K(4+C_{Q,\beta,d})}{M},
 \end{equation*}
 and
 \begin{equation*}
     \widetilde{q}_k(\bx)\geq \frac{1}{M},\ \ \forall \, k\in\{1,\cdots,K\},\, \forall \, \bx\in[0,1]^d.
 \end{equation*}
\end{Lemma}
\begin{proof}
Composing the neural networks in Theorem \ref{T:MSTLogNetworkResult} and Theorem \ref{T:ReluNWApproxResult} results in a neural network $\bG=(G(H_1),\cdots,G(H_K))$ such that for any $k=1,\cdots,K,$
\begin{equation*}
\begin{aligned}
\big\|e^{G(H_k)}-p_k^0\big\|_{\infty}
\leq \big\|e^{G(H_k)}-H_k\big\|_{\infty}+\big\|H_k-p_k^0\big\|_{\infty}
\leq \frac{4+C_{Q,\beta,d}}{M}.
\end{aligned}
\end{equation*}
Define now the vector valued function $\widetilde{\bq}$ component-wise by
\begin{equation*}
  \widetilde{q}_k(\bx)=\frac{e^{G(H_k(\bx))}}{\sum_{j=1}^Ke^{G(H_j(\bx))}}, \ \ k=1,\cdots,K .
\end{equation*}
Applying the composition \eqref{Eq: Composition}, depth synchronization \eqref{Eq: Depth synchronization} and parallelization rules \eqref{Eq: Parallelization} it follows that $\widetilde{\bq}\in \mF_{\bPhi}(L,\bfm,s)$.
To bound $\|\widetilde{q}_k-p_k^0\|_{\infty}$, we use that $\bp_0=(p_1^0,\cdots,p_K^0)$ is a probability vector, $e^{G(H_j)}\geq 0$ for $j=1,\cdots,k$ and triangle inequality, to obtain
\begin{equation*}
\begin{aligned}
\left\|\widetilde{q}_k-p_k^0\right\|_{\infty}\leq& \left\|e^{G(H_k)}\left(\frac{1}{\sum_{j=1}^Ke^{G(H_j)}}-1\right)\right\|_{\infty}+\left\|e^{G(H_k)}-p_k^0\right\|_{\infty}\\
=&\left\|e^{G(H_k)}\left(\frac{\sum_{\ell=1}^Kp_{\ell}^0}{\sum_{j=1}^Ke^{G(H_j)}}-\frac{\sum_{\ell=1}^Ke^{G(H_{\ell})}}{\sum_{j=1}^Ke^{G(H_j)}}\right)\right\|_{\infty}+\left\|e^{G(H_k)}-p_k^0\right\|_{\infty}\\
\leq& \Big(\sum_{\ell=1}^K\left\|p_{\ell}^0-e^{G(H_{\ell})}\right\|_{\infty}\Big) \left\|\frac{e^{G(H_k(\cdot))}}{\sum_{j=1}^Ke^{G(H_j)}}\right\|_{\infty}+ \left\|e^{G(H_k)}-p_k^0\right\|_{\infty}\\
\leq& \frac{(K+1)(4+C_{Q,\beta,d})}{M}\leq \frac{2K(4+C_{Q,\beta,d})}{M}.
\end{aligned}
\end{equation*}
For the second bound of the lemma, notice that from the first bound of the lemma and the second bound of Theorem \ref{T:MSTLogNetworkResult} it follows that
\begin{equation*}
\widetilde{q}_k(\bx)\geq   \frac{\frac{4}{M}}{\sum_{j=1}^Ke^{G(H_j(\bx))}}\geq \frac{\frac{4}{M}}{1+K\frac{(4+C_{Q,\beta,d})}{M}}=\frac{4}{M+K(4+C_{Q,\beta,d})}\geq \frac{1}{M},
\end{equation*}
where for the second inequality we used that $p_j(\bx)\leq 1$, so $e^{G(H_j(\bx))}\leq p_j^0(\bx)+(4+C_{Q,\beta,d})/M$ and for the last inequality we used that $M\geq K(4+C_{Q,\beta,d}).$
\end{proof}

The Kullback-Leibler divergence can be upper bounded by the $\chi^2$-divergence, see for instance Lemma 2.7 in \cite{tsybakov2008introduction}. Thus,
\begin{equation*}
\mathbb{E}_{\bX}\left[(\bp_0(\bX))^\top\log\left(\frac{\bp_0(\bX)}{\widetilde{\bq}(\bX)}\right)\right]\leq \mathbb{E}_{\bX}\left[\sum_{k=1}^K\frac{(p_k^0(\bX)-\widetilde{q}_k(\bX))^2}{\widetilde{q}_k(\bX)}\right].
\end{equation*}
To control the approximation error, we can combine this bound with the first bound of Lemma \ref{L:ApproximationSoftmaxNetwork} to conclude that if $p_k^0(\bX)>2K(4+C_{Q,\beta,d})/M,$ then
\begin{equation*}
    \frac{(p_k^0(\bX)-\widetilde{q}_k(\bX))^2}{\widetilde{q}_k(\bX)}\leq \frac{4K^2(4+C_{Q,\beta,d})^2}{M^2}\left(p_k^0(\bX)-\frac{2K(4+C_{Q,\beta,d})}{M}\right)^{-1}.
\end{equation*}
On the other hand, combining the bound with the second inequality from the same lemma yields
\begin{equation*}
    \frac{(p_k^0(\bX)-\widetilde{q}_k(\bX))^2}{\widetilde{q}_k(\bX)}\leq\sum_{k=1}^K\frac{4K^2(4+C_{Q,\beta,d})^2}{M^2}\left(\max\Big\{p_k^0(\bX)-\frac{2K(4+C_{Q,\beta,d})}{M},\frac{1}{M}\Big\}\right)^{-1},
\end{equation*}
which is valid for all possible values of $\bp_0(\bx)\in[0,1]^k$. As $M$ tends to infinity, $p_k^0(\bx)-2K(4+C_{Q,\beta,d})/M$ tends to $p_k^0(\bx)$, while $1/M$ tends to zero. Without any further conditions on $p_k^0(\bX)$ this bound is thus of order $M^{-1}.$ The small value bound, however, allows us to obtain an upper bound with better behaviour in $M.$ The following proposition employs the small value bound to control the expectation of $(p_k^0(\bx))^{-1}$ on the set that $p_k^0(\bx)$ exceeds some threshold value $H.$

\begin{propositie}\label{P:UpperBoundIk}
Assume there exists an $\alpha\geq 0$ and a finite constant $C<\infty,$ such that for $\bp =(p_1,\dots,p_K):\mathcal{D}\rightarrow \mathcal{S}^K$ we have $\mathbb{P}_{\bX}(p_k(\bX)\leq t)\leq Ct^{\alpha}$ for all $t\geq 0$ and $k\in \{1,\dots, K\}.$ Let $H\in[0,1]$. Then it holds that
\begin{equation*}
    \int_{\{p_k(\bx)\geq H\}}\frac{1}{p_k(\bx)} \, d\mathbb{P}_{\bX}(\bx)\leq\begin{cases}
    C\frac{H^{\alpha-1}}{1-\alpha}, & \text{ if } \alpha\in[0,1),\\
    C(1-\log(H)), & \text{ if } \alpha\geq 1.
    \end{cases}
\end{equation*}
\end{propositie}
\begin{proof} 
Observe that $p_k(\bX)$ is a probability. Therefore, $p_k(\bX)\leq 1$ and consequently $C\geq 1.$ For any nonnegative function $h$ and random variable $Z \sim \mathbb{P}_Z,$ we have
$\int h(Z)\,d\mathbb{P}_Z=\mathbb{E}[h(Z)]=\int_0^{\infty} \mathbb{P}_Z(h(Z)\geq u) \, du.$
Hence 
\begin{equation*}
    \begin{aligned}
    \int_{\{p_k(\bx)\geq H\}}\frac{1}{p_k(\bx)} \, d\mathbb{P}_{\bX}(\bx)=\int_0^{\infty} \mathbb{P}_{\bX}\left(\frac{1}{p_k(\bX)}\1_{\{p_k(\bX)\geq H\}}\geq u\right)\,du
    \leq \int_0^{\frac{1}{H}} \mathbb{P}_{\bX}\Big(p_k(\bX)\leq \frac{1}{u}\Big)\, du,
    \end{aligned}
\end{equation*}
where the inequality follows from observing that $\frac{1}{p_k(\bX)}\1_{\{p_k(\bX)\geq H\}}\geq u$ implies $H\leq p(\bX)\leq \frac{1}{u}$ and  $u\leq 1/H$.

If $\alpha=0$, we use the trivial bound $\mathbb{P}_{\bX}(p_k(\bx)\leq t)\leq 1,$ for all $t\in[0,1],$ and obtain
\begin{equation*}
    \int_0^{\frac{1}{H}} \mathbb{P}_{\bX}\Big(p_k(\bX)\leq \frac{1}{u}\Big)\, du\leq \int_0^{\frac{1}{H}} 1\, du = \frac{1}{H}.
\end{equation*}
If $0<\alpha<1,$ we can invoke the assumption of this proposition to obtain
\begin{equation*}
    \int_0^{\frac{1}{H}} \mathbb{P}_{\bX}\Big(p_k(\bX)\leq \frac{1}{u}\Big)\, du\leq C\int_0^{\frac{1}{H}} u^{-\alpha}\, du=\frac{CH^{\alpha-1}}{1-\alpha}.
\end{equation*}
For $\alpha \geq 1,$ we have $\mathbb{P}_{\bX}(p_k(\bX)\leq t)\leq Ct$ for all $0 \leq t\leq 1.$ If moreover $C\leq H^{-1}$, the inequality $\mathbb{P}_{\bX}(p_k(\bX)\leq t)\leq \min\{1,Ct\}$ leads to
\begin{equation*}
\begin{aligned}
    \int_0^{\frac{1}{H}} \mathbb{P}_{\bX}\Big(p_k(\bX)\leq \frac{1}{u}\Big)\, du&\leq \int_0^{C} 1\, du+C\int_C^{\frac{1}{H}} \frac 1u \, du
    = C+C(-\log(H)-\log(C)).
\end{aligned}
\end{equation*}
If $\alpha\geq 1$ and $C\geq H^{-1},$ we can upper bound the integral by $\int_0^{C} 1\, du = C.$ The result of the proposition now follows from simplifying the expressions using that $C\geq 1$.
\end{proof}

We can now state and prove the main approximation bound.

\begin{proof}[Proof of Theorem \ref{S:ApproxUnderCondA}]
The condition $\|\bp-\bp_0\|_{\infty}\leq C_1/M$ implies that $p_k(\bx)\geq p^0_k(\bx)-C_1/M.$ Combined with $p_k(\bx)\geq 1/M,$ this gives
\begin{equation*}
    p_k(\bx)\geq \Big(p^0_k(\bx)-\frac {C_1}M\Big) \vee \frac 1M \geq \frac{p^0_k(\bx)}{C_1+1} \vee \frac 1M,
\end{equation*} where we used that $p_k^0(\bx)\geq (C_1+1)/M=((C_1+1)/C_1) \cdot (C_1/M)$ implies
\begin{equation*}
    p_k^0(\bx)-\frac{C_1}{M}\geq p_k^0(\bx)\left(1-\frac{C_1}{C_1+1}\right)=\frac{p_k^0(\bx)}{C_1+1}.
\end{equation*} This gives rise to the upper bound \begin{equation*}
    \frac{(p_k^0(\bX)-p_k(\bX))^2}{p_k(\bX)}\leq \frac{C_1^2}{M}\1_{\{p^0_k(\bx)\leq \frac{C_1+1}{M}\}}+\frac{C_1^2}{M^2}\cdot\frac{C_1+1}{p^0_k(\bx)}\1_{\{p^0_k(\bx)\geq \frac{C_1+1}{M}\}}.
\end{equation*}
Taking the expectation over the right hand side yields
\begin{equation*}
     \frac{C_1^2}{M}\mathbb{P}_{\bX}\left(p^0_k(\bx)\leq \frac{C_1+1}{M}\right)+\frac{C_1^2 (C_1+1)}{M^2}\int_{\{p^0_k(\bx)\geq \frac{C_1+1}{M}\}}\frac{1}{p^0_k(\bx)}\,d\mathbb{P}_{\bX}(\bx)
\end{equation*}
By the $\alpha$-SVB condition the first term is upper bounded by
\begin{equation*}
    \frac{C_1^2}{M}\mathbb{P}_{\bX}\left(p^0_k(\bx)\leq \frac{C_1+1}{M}\right)\leq  \frac{C_1^2C}{M}\left(\frac{C_1+1}{M}\right)^{\alpha \wedge 1}
    \leq C \frac{(C_1+1)^{2+(\alpha \wedge 1)}}{M^{1+(\alpha \wedge 1)}}.
\end{equation*}
Applying Proposition \ref{P:UpperBoundIk} with $H=(C_1+1)/M$ to the second term yields the result.
\end{proof}

Now we have all the ingredients to complete the proof of the main theorem.

\begin{proof}[Proof of Theorem \ref{S: Main Risk Bound}] Take $\delta=n^{-1}$ and $\epsilon=C_n=1$ in
Theorem \ref{S: Main Oracle Inequality}. Using that $d_{\tau}$ is upper bounded by the sup-norm distance together with Lemma \ref{L: Covering Number bound} gives
\begin{equation}\label{eq.oracle_recall}
\begin{aligned}
    R_B(\bp_0,\widehat \bp)&\leq 2\left(\inf_{\bp\in\mF}R(\bp_0,\bp)+\Delta_n(\bp_0,\widehat{\bp})+\frac{3}{n}\right)\\
&+4\cdot\frac{68B(s+1)\log(2^{2L+6}n(L+1)K^3d^2s^L)+272B+(3/2)K(\log(n)+B)}{n}.
\end{aligned}
\end{equation}
Recall that $0\leq \alpha \leq 1$ is the index from the SVB condition. We now choose $M= \lfloor c K^{\frac{(2+\alpha)\beta}{(1+\alpha)\beta+d}} n^{\frac{\beta}{(1+\alpha)\beta+d}} \rfloor$ for a small constant $c$ chosen below. To apply Lemma \ref{L:ApproximationSoftmaxNetwork}, we need to show that $M \gg K.$ To see this, observe that $R_B(\bp_0,\widehat \bp)\leq B$ and therefore the convergence rate becomes trivial if  $\phi_n \geq 1.$ Using that $\phi_n=K^{\frac{(1+\alpha)\beta+(3+\alpha)d}{(1+\alpha)\beta+d}} n^{-\frac{(1+\alpha)\beta}{(1+\alpha)\beta+d}},$ this implies $K\leq n^{\frac{(1+\alpha)\beta}{(1+\alpha)\beta+(3+\alpha)d}}\leq n^{\frac{\beta}{\beta+2d}}\leq n^{\frac{\beta}{2d}}.$ Hence, $K^{d-\beta}\ll n^\beta$ and thus also $M \gg K.$

For this choice of $M,$ the network $\widetilde{\bq}$ from Lemma \ref{L:ApproximationSoftmaxNetwork} is in the network class $\mF_{\bPhi}(L,\bfm,s),$ where $L=3\ceil{\log_2(M)(d/\beta+1)}(1+\ceil{\log_2(d+\beta)})+\floor{40(\beta+2)^2\log_2(M)}+2,$ the maximum width of the hidden layers is bounded by $\lesssim Kc^{d/\beta}M^{d/\beta}=c^{d/\beta}n\phi_n$ and similarly $s\lesssim Kc^{d/\beta} M^{d/\beta}\log_2(M)=c^{d/\beta}n\phi_n\log_2(M).$ In particular, by taking $c$ sufficiently small and using the depth synchronization property \eqref{Eq: Depth synchronization}, $\widetilde{\bq}\in \mF_{\bPhi}(L,\bfm,s),$ whenever $A(d,\beta)\log_2(n)\leq L\lesssim n\phi_n,$ for a suitable constant $A(d,\beta),$ the maximum width is $\gtrsim n\phi_n$ and $s\asymp n\phi_n \log(n).$ We now apply Theorem \ref{S:ApproxUnderCondA} with $C_1=2K(4+C_{Q,\beta,d}).$ Using that $C_1+1=2K(4+C_{Q,\beta,d})+1\leq 2K(5+C_{Q,\beta,d}),$ we find
\begin{equation*}
\begin{aligned}
    \inf_{\bp\in\mF}R(\bp_0,\bp)
    \leq 
    8CK^{3+\alpha} \frac{(5+C_{Q,\beta,d})^3}{M^{1+\alpha}}
    \Big(
    1+\frac{\1_{\{\alpha<1\}}}{1-\alpha}
    +\log(M)\Big) \lesssim \phi_n \log(n).
\end{aligned}
\end{equation*}
Together with \eqref{eq.oracle_recall} and $s\asymp n\phi_n \log(n),$ the statement of Theorem \ref{S: Main Risk Bound} follows.
\end{proof}

\subsection{Oracle inequality related results}

In this section we prove Theorem \ref{S: Main Oracle Inequality}. For $B>0$, consider
\begin{align*}
	R_{B,n}(\bp_0,\widehat \bp):=\E_{\mD_n} \Big[\frac{1}{n}\sum_{i=1}^n\bY^\top_i\Big(B\wedge \log\Big(\frac{\bp_0(\bX_i)}{\widehat \bp(\bX_i)} \Big)\Big].
\end{align*}
The next proposition shows how this risk is related to the approximation error and the quantity $\Delta_n(\bp_0,\widehat{\bp})$ defined in \eqref{Eq: Difference Expected Empirical Risk Estimator vs Minimum} that measures the empirical distance between an arbitrary estimator and an empirical risk minimizer.

\begin{propositie}\label{P: Upper Bound Risk Delta}
For any estimator $\widehat{\bp}\in\mF$,
$$R_{B,n}(\bp_0,\widehat \bp)\leq R_{\infty,n}(\bp_0,\widehat{\bp})\leq \inf_{\bp\in\mF}R(\bp_0,\bp)+\Delta_n(\bp_0,\widehat{\bp}).$$
\end{propositie}
\begin{proof}
The first inequality follows from $a\geq \min(a,b)$, for all $a,b\in\mathbb{R}$. To prove the second inequality, fix a $\bp^*\in\mF$. Using that $\Delta_n(\bp_0,\bp^*)\geq 0$ and $$\mathbb{E}_{\mD_n}\big[\bY_i^\top\log(\bp^*(\bX_i))\big]=\mathbb{E}_{\mD_n}\big[\mathbb{E}_{\mD_n}[\bY_i^\top|\bX_i]\log(\bp^*(\bX_i))\big]=\mathbb{E}_{\mD_n}\big[\bp_0(\bX_i)^\top\log(\bp^*(\bX_i))\big],$$ we get
 \begin{equation*}
     \begin{aligned}
     \mathbb{E}_{\mD_n}\Big[-\frac{1}{n}\sum_{i=1}^n\bY_i^\top\log(\widehat{\bp}(\bX_i))\Big]&\leq\mathbb{E}_{\mD_n}\Big[-\frac{1}{n}\sum_{i=1}^n\bY_i^\top\log(\widehat{\bp}(\bX_i))\Big]+\Delta_n(\bp_0,\bp^*)\\
    &=\mathbb{E}_{\mD_n}\Big[-\frac{1}{n}\sum_{i=1}^n\bY_i^\top\log(\bp^*(\bX_i))\Big]+\Delta_n(\bp_0,\widehat{\bp}) \\
    &=\mathbb{E}_{\bX}\Big[-\bp_0^\top(\bX)\log(\bp^*(\bX))\Big]+\Delta_n(\bp_0,\widehat{\bp}).
     \end{aligned}
 \end{equation*}
As this holds for all $\bp^*\in\mF$, we can take on the right hand side also the infimum over all $\bp^*\in\mF.$ To complete the proof for the second inequality, we add to both sides $\mathbb{E}_{\mD_n}[\bY_i^\top\log(\bp_0(\bX_i))]=\mathbb{E}_{\mD_n}[\bp_0(\bX_i)^\top\log(\bp_0(\bX_i))].$
\end{proof}

The truncation level $B$ allows us to split the statistical risk into multiple parts that can be controlled separately. The following lemma provides a bound on the event that $p_k^0(\bX)$ is small.

\begin{Lemma}\label{L: Bound for case of small p0}
Let $\mF$ be a class of conditional class probabilities, $\widehat{\bp}$ be any estimator taking values in $\mF$, $(\overline \bX,\overline \bY)$ be a random pair with the same distribution as $(\bX_1,\bY_1)$ and $C_n \in (0,n/e]$. Then, for any $i\in\{1,\cdots,n\},$ and any $k\in\{1,\cdots,K\},$ we have
\begin{equation*}
\left|\mathbb{E}_{\mD_n,(\overline\bX,\overline\bY)}\left[\overline Y_{k}\1_{\{p_{k}^0(\overline\bX)\leq\frac{C_n}{n}\}}\Big( B \wedge \log\Big(\frac{p_{k}^0(\overline\bX)}{\widehat{p}_k(\overline\bX)}\Big)\Big)\right]\right|\leq \frac{C_n\big(\log\big(\frac{n}{C_n}\big)+B\big)}{n}.
\end{equation*}
\end{Lemma}
\begin{proof}
Since $\bp_0,\widehat{\bp}\in[0,1]^K$, we have
\begin{equation}\label{Eq: Double Bound for case small p0}
\log(p_{k}^0(\overline\bX))\leq B \wedge \log\Big(\frac{p_{k}^0(\overline\bX)}{\widehat{p}_k(\overline\bX)}\Big)\leq B.
\end{equation}
Using that $a\leq x\leq b$ implies $|x|\leq \max\{|a|,|b|\}\leq |a|+|b|$ and $Y_k\geq 0$, we can get an upper bound that does not depend on $\widehat{\bp}$
\begin{equation*}
\begin{aligned}
&\left|\mathbb{E}_{\mD_n,(\overline\bX,\overline\bY)}\left[\overline Y_{k}\1_{\{p_{k}^0(\overline\bX)\leq\frac{C_n}{n}\}}\Big( B \wedge \log\Big(\frac{p_{k}^0(\overline\bX)}{\widehat{p}_k(\overline\bX)}\Big)\Big)\right]\right|\\
&\leq \mathbb{E}_{(\overline\bX,\overline\bY)}\left[\overline Y_{k}\1_{\{p_{k}^0(\overline\bX)\leq\frac{C_n}{n}\}}\left|\log(p_{k}^0(\overline\bX))\right|\right]+\mathbb{E}_{(\overline\bX,\overline\bY)}\left[\overline Y_{k}\1_{\{p_{k}^0(\overline\bX)\leq\frac{C_n}{n}\}}B\right]\\
&=\mathbb{E}_{\overline\bX}\left[p_{k}^0(\overline\bX)\1_{\{p_{k}^0(\overline\bX)\leq\frac{C_n}{n}\}}\left|\log(p_{k}^0(\overline\bX))\right|\right]+\mathbb{E}_{\overline\bX}\left[p_{k}^0(\overline\bX)\1_{\{p_{k}^0(\overline\bX)\leq\frac{C_n}{n}\}}B\right],
\end{aligned}
\end{equation*}
where the last equality follows from conditioning on $\overline\bX$. Using that the function $u\mapsto u|\log(u)|$ is monotone increasing on $(0,e^{-1})$ and $n\geq eC_n$, yields
\begin{equation*}
\left|\mathbb{E}_{\mD_n,(\overline\bX,\overline\bY)}\left[\overline Y_{k}\1_{\{p_{k}^0(\overline\bX)\leq\frac{C_n}{n}\}}\Big( B \wedge \log\Big(\frac{p_{k}^0(\overline\bX)}{\widehat{p}_k(\overline\bX)}\Big)\Big)\right]\right|\leq \frac{C_n\big(\log\big(\frac{n}{C_n}\big)+B\big)}{n}.
\end{equation*}
\end{proof}
\begin{Corollary}\label{C: Upper and Lower Bound for case of small p0}
Under the conditions of Lemma \ref{L: Bound for case of small p0} it holds that
\begin{equation*}
-\frac{C_n\log(n/C_n)}{n}\leq\mathbb{E}_{\mD_n,(\overline\bX,\overline\bY)}\left[\overline Y_{k}\1_{\{p_{k}^0(\overline\bX)\leq\frac{C_n}{n}\}}\Big( B \wedge \log\Big(\frac{p_{k}^0(\overline\bX)}{\widehat{p}_k(\overline\bX)}\Big)\Big)\right]\leq \frac{C_nB}{n}.
\end{equation*}
\end{Corollary}
\begin{proof}
The lower and upper bound can be obtained from \eqref{Eq: Double Bound for case small p0}, $\overline Y_k\geq 0$ and the fact that $u\mapsto u\log(u)$ is monotone decreasing on $(0,e^{-1})$.
\end{proof}
Both Lemma \ref{L: Bound for case of small p0} and Corollary \ref{C: Upper and Lower Bound for case of small p0} do not require that the random pair $(\overline\bX,\overline\bY)$ is independent of the data. Specifically, they also hold in the case that $(\overline\bX,\overline\bY)=(\bX_i,\bY_i)$ for some $i\in\{1,\cdots,n\}$.

\begin{proof}[Proof of Theorem \ref{S: Main Oracle Inequality}]
For ease of notation set
$$\Big( B \wedge \log\Big(\frac{\bp_0(\bX_i)}{\widehat{\bp}(\bX_i)}\Big)\Big)_{\geq C_n/n}$$ to denote the vector with coefficients
$$\1_{\{p_{k}^0(\bX_i)\geq\frac{C_n}{n}\}} \Big( B \wedge \log\Big(\frac{p_{k}^0(\bX_i)}{\widehat{p}_k(\bX_i)}\Big)\Big), \quad k=1,\dots,K.$$
For i.i.d. random pairs $(\wbX_i,\wbY_i)$, $i=1,\cdots,n$ with joint distribution $\mathbb{P}$ that are generated independently of the data sample define $\mD_n':=\{(\bX_i,\bY_i)_i,(\wbX_i,\wbY_i)_i\}$. Then, for any $C_n>0$,
\begin{align}
\begin{aligned}
&|R_{B}(\bp_0,\widehat{\bp})-R_{B,n}(\bp_0,\widehat{\bp})|\\
&=\left|\mathbb{E}_{\mD_n'}\left[\frac{1}{n}\sum_{i=1}^n\sum_{k=1}^K \widetilde Y_{i,k}\Big( B \wedge \log\Big(\frac{p_{k}^0(\wbX_i)}{\widehat{p}_k(\wbX_i)}\Big)\Big)-\frac{1}{n}\sum_{i=1}^n\sum_{k=1}^K Y_{i,k} \Big( B \wedge \log\Big(\frac{p_{k}^0(\bX_i)}{\widehat{p}_k(\bX_i)}\Big)\Big)\right]\right|\\
&\leq  (I) + (II) + (III),
\end{aligned}\label{eq.I-III_decomp}
\end{align}
where
\begin{align*}
(I) &=\left|\mathbb{E}_{\mD_n'}\left[\frac{1}{n}\sum_{i=1}^n\left(\widetilde \bY_i^\top  \Big( B \wedge \log\Big(\frac{\bp_0(\widetilde \bX_i)}{\widehat{\bp}(\widetilde\bX_i)}\Big)\Big)_{\geq C_n/n}
- \bY_i^\top  \Big( B \wedge \log\Big(\frac{\bp_0( \bX_i)}{\widehat{\bp}(\bX_i)}\Big)\Big)_{\geq C_n/n}\right)\right]\right|\\
(II) &=\left|\mathbb{E}_{\mD_n'}\left[\frac{1}{n}\sum_{i=1}^n\sum_{k=1}^K \widetilde{Y}_{i,k}\1_{\{p_{k}^0(\wbX_i)\leq\frac{C_n}{n}\}}\Big( B \wedge \log\Big(\frac{p_{k}^0(\wbX_i)}{\widehat{p}_k(\wbX_i)}\Big)\Big)\right]\right|\\
(III)&=\left|\mathbb{E}_{\mD_n'}\left[\frac{1}{n}\sum_{i=1}^n\sum_{k=1}^K Y_{i,k}\1_{\{p_{k}^0(\bX_i)\leq\frac{C_n}{n}\}}\Big( B \wedge \log\Big(\frac{p_{k}^0(\bX_i)}{\widehat{p}_k(\bX_i)}\Big)\Big)\right]\right|.
\end{align*}
First we bound the terms (II) and (III). Applying Lemma \ref{L: Bound for case of small p0} in total $nK$ times with $(\overline\bX,\overline\bY)=(\widetilde \bX_i,\widetilde \bY_i)$,  yields
\begin{equation}
(II)\leq \frac{1}{n}\sum_{i=1}^n\sum_{k=1}^K \frac{C_n\big(\log\big(\frac{n}{C_n}\big)+B\big)}{n}=\frac{C_nK\big(\log\big(\frac{n}{C_n}\big)+B\big)}{n},
\label{eq.II_bd}
\end{equation}
while taking $(\overline\bX,\overline\bY)=(\bX_i,\bY_i)$ in Lemma \ref{L: Bound for case of small p0} yields
\begin{equation}
(III)\leq \frac{1}{n}\sum_{i=1}^n\sum_{k=1}^K \frac{C_n\big(\log\big(\frac{n}{C_n}\big)+B\big)+B)}{n}=\frac{C_nK\big(\log\big(\frac{n}{C_n}\big)+B\big)}{n}.
\label{eq.III_bd}
\end{equation}

Now we deal with the term (I). 
Due to the bound $B$ and the indicator function
\begin{equation}\label{Eq: Upper-lower bound property}
\1_{\{p_{k}^0(\bX_i)\geq\frac{C_n}{n}\}} \Big( B \wedge \log\Big(\frac{p_{k}^0(\bX_i)}{\widehat{p}_k(\bX_i)}\Big)\Big)=\1_{\{p_{k}^0(\bX_i)\geq\frac{C_n}{n}\}} \Big( B \wedge \log\Big(\frac{p_{k}^0(\bX_i)}{(C_ne^{-B}/n)\vee\widehat{p}_k(\bX_i)}\Big)\Big).
\end{equation}
Given a minimal (internal) $\delta$-covering of $\log(\mF)$ with respect to the pseudometric $d_{\tau}$, with $\tau=\log(C_ne^{-B}/n)$, denote the centers of the balls by $\bp_\ell$. Then there exists a random $\ell^*$ such that $$\Big\|\log\Big(\frac{C_ne^{-B}}n\Big)\vee\log(\widehat{\bp})-\log\Big(\frac{C_ne^{-B}}n\Big)\vee\log(\bp_{\ell^*})\Big\|_{\infty}\leq \delta.$$ 
This together with \eqref{Eq: Upper-lower bound property} and using that $\bY$ is one of the $K$-dimensional standard basis vectors yields
\begin{equation}\label{Eq: Split recombined}
(I)\leq \mathbb{E}_{\mD_n'}\left[\left|\frac{1}{n}\sum_{i=1}^nG_{\ell^*}(\wbX_i,\wbY_i,\bX_i,\bY_i)\right|\right]+2\delta,
\end{equation}
where $$G_{\ell^*}(\wbX_i,\wbY_i,\bX_i,\bY_i):= \wbY_i^\top\Big(B \wedge \log\Big( \frac{\bp_0(\wbX_i)}{\bp_{\ell^*}(\wbX_i)}\Big)\Big)_{\geq C_n/n}- \bY_i^\top\Big( B \wedge\log\Big( \frac{\bp_0(\bX_i)}{\bp_{\ell^*}(\bX_i)}\Big)\Big)_{\geq C_n/n}.$$ 
For all $\ell\in\{1,\cdots,\mathcal{N}_n\}$ define $G_{\ell}$ in the same way. Moreover, write
$$\mathbf{Z}_i:=(\wbX_i,\wbY_i,\bX_i,\bY_i).$$
In a next step, we apply Bernstein's inequality (Proposition \ref{P: Bernsteins Inequality}) to $(G_{\ell}(\mathbf{Z}_i))_{i=1}^n$. Using that $(\bX_i,\bY_i)$ and $(\wbX_i,\wbY_i)$ have the same distribution, we get for the expectation of $G_{\ell}$ that
\begin{equation*}
\begin{aligned}
&\mathbb{E}_{\mD_n'}[G_{\ell}(\mathbf{Z}_i)]=0.
\end{aligned}
\end{equation*}
To verify the assumptions of Bernstein's inequality, it remains to prove that
\begin{equation}\label{Eq: Bernstein Moment Bound}
\mathbb{E}|G_{\ell}(\mathbf{Z}_i)|^m\leq m!(2B)^{m-2}R_B(\bp_0,\bp_{\ell})32B2^{-1}, \ \forall m\in\mathbb{N}_{\geq 2},
\end{equation}
such that, in the notation of Proposition \ref{P: Bernsteins Inequality}, we have $v_i=R_B(\bp_0,\bp_{\ell})32B$ and $U=2B$. To show this moment bound, observe that any real numbers $a,b$ satisfy $|a+b|^m\leq 2^m(|a|^m+|b|^m).$ Using moreover that $(\bX_i,\bY_i)$ and $(\wbX_i,\wbY_i)$ have the same distribution, the $m$-th absolute moment of $G_{\ell}$ is given by
\begin{equation*}
\begin{aligned}
&\mathbb{E}_{\mD_n'}\big[|G_{\ell}(\mathbf{Z}_i)|^m\big]\\
&=\mathbb{E}_{\mD_n'}\left[\left|\wbY_i^\top\left(B\wedge\log\left(\frac{\bp_0(\wbX_i)}{\bp_{\ell}(\wbX_i)}\right)\right)_{\geq C_n/n}-\bY_i^\top\left(B\wedge\log\left(\frac{\bp_0(\bX_{i})}{\bp_{\ell}(\bX_{i})}\right)\right)_{\geq C_n/n}\right|^m\right] \\
&\leq 2^{m+1}\mathbb{E}_{\mD_n}\left[\left|\bY_i^\top\left(B\wedge\log\left(\frac{\bp_0(\bX_{i})}{\bp_{\ell}(\bX_{i})}\right)\right)_{\geq C_n/n}\right|^m\right].
\end{aligned}
\end{equation*}
Triangle inequality gives
\begin{equation*}
\begin{aligned}
&\mathbb{E}_{\mD_n}\left[\left|\bY_i^\top\left(B\wedge\log\left(\frac{\bp_0(\bX_{i})}{\bp_{\ell}(\bX_{i})}\right)\right)_{\geq C_n/n}\right|^m\right]
\leq \mathbb{E}_{\mD_n} \left[\left(\bY_i^\top\left|\left(B\wedge\log\left(\frac{\bp_0(\bX_i)}{\bp_{\ell}(\bX_i)}\right)\right)_{\geq C_n/n}\right|\right)^m\right],
\end{aligned}
\end{equation*}
where for a vector $\mathbf{v}$, $|\mathbf{v}|$ denotes the absolute value coefficient-wise.
Since $\bY$ is one of the standard basis vectors, it holds that $Y_k\in\{0,1\}$, and $Y_kY_j$ is equal to $0$ when $j\neq k$ and equal to $Y_k$ when $k=j$. Using this observation together with conditioning on $\bX_i$ yields
\begin{equation*}
\begin{aligned}
&\mathbb{E}_{\mD_n} \left[\left(\bY_i^\top\left|\left(B\wedge\log\left(\frac{\bp_0(\bX_i)}{\bp_{\ell}(\bX_i)}\right)\right)_{\geq C_n/n}\right|\right)^m\right]\\
&=\mathbb{E}_{\mD_n} \left[\bY_i^\top\left|\left(B\wedge\log\left(\frac{\bp_0(\bX_i)}{\bp_{\ell}(\bX_i)}\right)\right)_{\geq C_n/n}\right|^m\right]\\
&=\mathbb{E}_{\bX_i}\left[\bp_0^\top(\bX_i)\left|\left(B\wedge\log\left(\frac{\bp_0(\bX_i)}{\bp_{\ell}(\bX_i)}\right)\right)_{\geq C_n/n}\right|^m\right]\\ &\leq \mathbb{E}_{\bX_i}\left[\bp_0^\top(\bX_i)\left|B\wedge\log\left(\frac{\bp_0(\bX_i)}{\bp_{\ell}(\bX_i)}\right)\right|^m\right],
\end{aligned}
\end{equation*}
where we used for the last inequality that for every set $\Omega$, each $A\subseteq \Omega$, every function $\theta:\Omega\rightarrow \mathbb{R}$ and every $m\in\mathbb{N}_{\geq 2}$ it holds that
$|\1_{A}\theta|^m=(\1_A)^m|\theta|^m=\1_A|\theta|^m\leq |\theta|^m.$ Combining the previous displays and applying Lemma \ref{L: Inequality to relate to the risk}, we get that
\begin{equation}\label{Eq: Upper bound for Moments pre-Bernstein notation}
\begin{aligned}
&\mathbb{E}_{\mD_n'}[|G_{\ell}(\mathbf{Z}_i)|^m]\\
&\leq 2^{m+1}\mathbb{E}_{\bX_i}\left[\bp_0^\top(\bX_{i})\left|B\wedge\log\left(\frac{\bp_0(\bX_{i})}{\bp_{\ell}(\bX_{i})}\right)\right|^m\right]\\
&\leq  2^{m+1}C_{m,B}\mathbb{E}_{\bX_i}\left[\bp_0^\top(\bX_{i})\left(B\wedge\log\left(\frac{\bp_0(\bX_{i})}{\bp_{\ell}(\bX_{i})}\right)\right)\right] =2^{m+1}C_{m,B}R_B(\bp_0,\bp_{\ell}),
\end{aligned}
\end{equation}
where $C_{m,B}$ is given by
$$C_{m,B}=\max\left\{m!,\frac{B^m}{B-1}\right\}.$$ Since $B\geq 2$, we get that $B/(B-1)\leq 2$ and $C_{m,B}\leq \max\left\{m!,2B^{m-1}\right\}\leq 2m!B^{m-1}.$ Together with \eqref{Eq: Upper bound for Moments pre-Bernstein notation} this yields
\begin{equation*}
\mathbb{E}_{\mD_n'}[|G_{\ell}(\mathbf{Z}_i)|^m]\leq 2^{m+1}C_{m,B}R_B(\bp_0,\bp_{\ell})\leq m!(2B)^{m-2}R_B(\bp_0,\bp_{\ell})32B2^{-1},
\end{equation*}
completing the proof for the moment bound \eqref{Eq: Bernstein Moment Bound}.

Now define $z_{\ell}:=\sqrt{n^{-1}68B\log(\mathcal{N}_n)}\vee \sqrt{\mathbb{E}_{(\bX,\bY)}[\bY^\top(B\wedge\log(\bp_0(\bX)/\bp_{\ell}(\bX)))]}.$ Since $B\geq 2$, Lemma \ref{lem.Hell_KL_bd} guarantees that the truncated Kullback-Leibler risk is always nonnegative, so $z_{\ell}$ is well defined. Define $z^*=z_{\ell^*}$, that is,
$$z^*=\sqrt{\frac{68B\log(\mathcal{N}_n)}{n}}\vee \sqrt{\mathbb{E}_{\mD_N,(\bX,\bY)}\Bigg[\bY^\top\Bigg(B\wedge\log\Bigg(\frac{\bp_0(\bX)}{\bp_{\ell^*}(\bX)}\Bigg)\Bigg)\Bigg|\mD_n\Bigg]},$$ where we also condition on the dataset $\mathcal{D}_n$.
To upper bound $z^*$, we split the truncated empirical risk
\begin{equation*}
\begin{aligned}
&\mathbb{E}_{\mD_N,(\bX,\bY)}\Bigg[\bY^\top\Bigg(B\wedge\log\Bigg(\frac{\bp_0(\bX)}{\bp_{\ell^*}(\bX)}\Bigg)\Bigg)\Bigg|\mD_n\Bigg]\\
&=\mathbb{E}_{\mD_N,(\bX,\bY)}\left[\sum_{k=1}^KY_k\left(\1_{\{p_{k}^0(\bX)\leq\frac{C_n}{n}\}}+\1_{\{p_{k}^0(\bX)\geq\frac{C_n}{n}\}}\right)\Bigg(B\wedge \log\Bigg(\frac{p_{k}^0(\bX)}{p_{\ell^*,k}(\bX)}\Bigg)\Bigg)\Bigg|\mD_n\right].
\end{aligned}
\end{equation*}
Using the property of the $\delta$-cover, Equation \eqref{Eq: Upper-lower bound property} and the fact that $\bY$ is a standard basis vector, it holds that
\begin{equation*}
\begin{aligned}
&\mathbb{E}_{\mD_N,(\bX,\bY)}\left[\sum_{k=1}^KY_k\1_{\{p_{k}^0(\bX)\geq\frac{C_n}{n}\}}\Bigg(B\wedge\log\Bigg(\frac{p_{k}^0(\bX)}{p_{\ell^*,k}(\bX)}\Bigg)\Bigg)\Bigg|\mD_n\right]\\
&\leq \mathbb{E}_{\mD_N,(\bX,\bY)}\left[\sum_{k=1}^KY_k\1_{\{p_{k}^0(\bX)\geq\frac{C_n}{n}\}}\Bigg(B\wedge\log\Bigg(\frac{p_{k}^0(\bX)}{\widehat{p}_k(\bX)}\Bigg)\Bigg)\Bigg|\mD_n\right]+\delta.
\end{aligned}
\end{equation*}
On the other hand, applying Corollary \ref{C: Upper and Lower Bound for case of small p0}, with $(\overline \bX,\overline \bY)=(\bX,\bY)$, $K$ times for $\widehat{\bp}$ and $K$ times with $\widehat{\bp}$ replaced by $\bp_{\ell^*}$, yields
\begin{equation*}
\begin{aligned}
&\mathbb{E}_{\mD_N,(\bX,\bY)}\left[\sum_{k=1}^KY_k\1_{\{p_{k}^0(\bX)\leq\frac{C_n}{n}\}}\Bigg(B\wedge\log\Bigg(\frac{p_{k}^0(\bX)}{p_{\ell^*,k}(\bX)}\Bigg)\Bigg)\Bigg|\mD_n\right]\\
&\leq \mathbb{E}_{\mD_N,(\bX,\bY)}\left[\sum_{k=1}^KY_k\1_{\{p_{k}^0(\bX)\leq\frac{C_n}{n}\}}\Bigg(B\wedge\log\Bigg(\frac{p_{k}^0(\bX)}{\widehat{p}_{k}(\bX)}\Bigg)\Bigg)\Bigg|\mD_n\right]+\frac{C_nK\big(\log\big(\frac{n}{C_n}\big)+B\big)}{n}.
\end{aligned}
\end{equation*}
Define 
$$V:=\sqrt{\mathbb{E}_{\mD_N,(\bX,\bY)}\Bigg[\bY^\top\Bigg(B\wedge\log\Bigg(\frac{\bp_0(\bX)}{\widehat{\bp}(\bX)}\Bigg)\Bigg)\Bigg|\mD_n\Bigg]}.$$
Combining the previous inequalities, we get that
\begin{equation*}
\begin{aligned}
\sqrt{\mathbb{E}_{\mD_N,(\bX,\bY)}\Bigg[\bY^\top(B\wedge\log\Bigg(\frac{\bp_0(\bX)}{\bp_{\ell^*}(\bX)}\Bigg)\Bigg)\Bigg|\mD_n\Bigg]}
&\leq V+\sqrt{\delta+\frac{C_nK\big(\log\big(\frac{n}{C_n}\big)+B\big)}{n}},
\end{aligned}
\end{equation*}
where we also used the elementary inequality $\sqrt{a+b}\leq \sqrt{a}+\sqrt{b}$ for all $a,b\geq 0$. Hence,
\begin{align}
    z^*\leq \sqrt{\frac{68B\log(\mathcal{N}_n)}{n}}+V+\sqrt{\delta+\frac{C_nK\big(\log\big(\frac{n}{C_n}\big)+B\big)}{n}}.
    \label{eq.zstar_bd}
\end{align}
The term $\sqrt{n^{-1}68B\log(\mathcal{N}_n)}$ is chosen such that in \eqref{Eq: E[T] inequality} and \eqref{Eq: E[T^2] inequality} below the equations balance out. Now define
$$T:=\max_{\ell}\left|\sum_{i=1}^n\frac{G_{\ell}(\mathbf{Z}_i)}{z_{\ell}}\right|.$$ The Cauchy-Schwarz inequality gives us that $\mathbb{E}_{\mD_n'}[VT]\leq \sqrt{\mathbb{E}_{\mD_n'}[V^2]\mathbb{E}_{\mD_n'}[T^2]}.$ Noticing that $\mathbb{E}_{\mD_n'}[V^2]=R_B(\bp_0,\widehat\bp)$, we get from \eqref{eq.I-III_decomp}, \eqref{eq.II_bd}, \eqref{eq.III_bd}, \eqref{Eq: Split recombined} and \eqref{eq.zstar_bd} that
\begin{equation}\label{Eq: Inequality Risk and Expectation T}
\begin{aligned}
\left|R_B(\bp_0,\widehat{\bp})-R_{B,n}(\bp_0,\widehat{\bp})\right|&\leq \frac{1}{n}\sqrt{R_B(\bp_0,\widehat{\bp})}\sqrt{\mathbb{E}_{\mD_n'}[T^2]}\\
&+ \frac{1}{n}\Bigg(\sqrt{\frac{68B\log(\mathcal{N}_n)}{n}}+\sqrt{\delta+\frac{C_nK\big(\log\big(\frac{n}{C_n}\big)+B\big)}{n}}\Bigg)\mathbb{E}_{\mD_n'}[T]\\
&+2\delta+\frac{2C_nK\big(\log\big(\frac{n}{C_n}\big)+B\big)}{n}.
\end{aligned}
\end{equation}
The next step in the proof derives bounds on $\mathbb{E}_{\mD_n'}[T]$ and $\mathbb{E}_{\mD_n'}[T^2]$. Using an union bound it holds that 
\begin{equation*}
\begin{aligned}
\mathbb{P}\left(T\geq t\right)
&=\mathbb{P}\left(\max_{\ell}\left|\sum_{i=1}^n\frac{G_{\ell}(\mathbf{Z}_i)}{z_{\ell}}\right|\geq t\right)= \mathbb{P}\left(\bigcup_{\ell=1}^{\mathcal{N}_n}\left(\left|\sum_{i=1}^n\frac{G_{\ell}(\mathbf{Z}_i)}{z_{\ell}}\right|\geq t\right)\right)\\
&\leq \sum_{\ell=1}^{\mathcal{N}_n}\mathbb{P}\left(\left|\sum_{i=1}^nG_{\ell}(\mathbf{Z}_i)\right|\geq tz_{\ell}\right).
\end{aligned}
\end{equation*}

We already showed that $G_{\ell}(\mathbf{Z}_i)$ satisfies the conditions of Bernstein's inequality (Proposition \ref{P: Bernsteins Inequality}) with $v_i=R_B(\bp_0,\bp_{\ell})32B$ and $U=2B$. Bernstein's inequality applied to the last term gives
\begin{equation*}
\begin{aligned}
&\mathbb{P}\left(T\geq t\right)\leq\sum_{\ell=1}^{\mathcal{N}_n}\mathbb{P}\left(\left|\sum_{i=1}^nG_{\ell}(\mathbf{Z}_i)\right|\geq tz_{\ell}\right)\\
&\leq \sum_{\ell=1}^{\mathcal{N}_n}2\exp\left(-\frac{(tz_{\ell})^2}{2nR_B(\bp_0,\bp_{\ell})32B+4Btz_{\ell}}\right)=2\mathcal{N}_n\exp\left(-\frac{t^2}{2n\frac{R_B(\bp_0,\bp_{\ell})32B}{z_{\ell}^2}+4B\frac{t}{z_{\ell}}}\right).
\end{aligned}
\end{equation*}
Since $z_{\ell}\geq \sqrt{R_B(\bp_0,\bp_{\ell})}$ it holds that $z_{\ell}^2\geq R_B(\bp_0,\bp_{\ell})$.
As probabilities are in the interval $[0,1],$ this gives us that
$$\mathbb{P}\left(T\geq t\right)\leq 1\wedge 2\mathcal{N}_n\exp\left(-\frac{t^2}{64Bn+4B\frac{t}{z_{\ell}}}\right).$$

If $t\geq \sqrt{68Bn\log(\mathcal{N}_n)}$, then since $z_{\ell}\geq \sqrt{n^{-1}68B\log(\mathcal{N}_n)}$ it holds that
$$\exp\left(-\frac{t^2}{64Bn+4B\frac{t}{z_{\ell}}}\right)\leq \exp\left(-\frac{t\sqrt{\log(\mathcal{N}_n)}}{\sqrt{68Bn}}\right).$$
For every nonnegative random variable $X$ with finite expectation one has $\mathbb{E}[X]=\int_0^{\infty}\mathbb{P}(X\geq t)\, dt.$ Therefore,
\begin{equation}\label{Eq: E[T] inequality}
\begin{aligned}
\mathbb{E}_{\mD_n'}[T]&\leq \sqrt{68Bn\log(\mathcal{N}_n)}+\int_{\sqrt{68Bn\log(\mathcal{N}_n)}}^{\infty}2\mathcal{N}_n\exp\left(-\frac{t\sqrt{\log(\mathcal{N}_n)}}{\sqrt{68Bn}}\right)\, dt\\
&=\sqrt{68Bn\log(\mathcal{N}_n)}+\sqrt{\frac{272Bn}{\log(\mathcal{N}_n)}}.
\end{aligned}
\end{equation}
Since $T$ is nonnegative, $\mathbb{P}(T^2\geq u)=\mathbb{P}(T\geq \sqrt{u})$, so using the same arguments as before we get that
\begin{equation*}
\begin{aligned}
\mathbb{E}_{\mD_n'}[T^2]&\leq 68Bn\log(\mathcal{N}_n)+\int_{68Bn\log(\mathcal{N}_n)}^{\infty}2\mathcal{N}_n\exp\left(-\sqrt{\frac{u\log(\mathcal{N}_n)}{68Bn}}\right)\, du.
\end{aligned}
\end{equation*}
Substitution $s=\sqrt{u}$ and integration by parts gives us that $\int_a^{\infty}e^{-\sqrt{u}b}\, du=2\int_{\sqrt{a}}^{\infty}se^{-sb}\, ds=2(\sqrt{a}b+1)e^{-\sqrt{a}b}/b^2$ and consequently
\begin{equation}\label{Eq: E[T^2] inequality}
\begin{aligned}
\mathbb{E}_{\mD_n'}[T^2]\leq 68Bn\log(\mathcal{N}_n)+544Bn,
\end{aligned}
\end{equation}
where we also used that $\mathcal{N}_n\geq e$ and thus $(\log(\mathcal{N}_n)+1)/\log(\mathcal{N}_n)\geq 2.$\\

Combining \eqref{Eq: E[T] inequality}, \eqref{Eq: E[T^2] inequality} with \eqref{Eq: Inequality Risk and Expectation T}, using twice that $2xy\leq x^2+y^2$ for all real numbers $x,y$, and using that $\log(\mathcal{N}_n)\geq 1$, we get that
\begin{equation}\label{Eq: Pre-epsilon-aid equation after substitution bounds on E[T] and E[T^2]}
\begin{aligned}
\left|R_B(\bp_0,\widehat{\bp})-R_{B,n}(\bp_0,\widehat{\bp})\right|&\leq \sqrt{R_B(\bp_0,\widehat{\bp})}\sqrt{\frac{68B\log(\mathcal{N}_n)+544B}{n}}\\
&+\frac{102B\log(\mathcal{N}_n)+272B}{n}+ 3\delta+\frac{3C_nK\big(\log\big(\frac{n}{C_n}\big)+B\big)}{n}.
\end{aligned}
\end{equation}

Setting $a=R_B(\bp_0,\widehat{\bp})$, $b=R_{B,n}(\bp_0,\widehat{\bp})$,
$$c=\sqrt{\frac{17B\log(\mathcal{N}_n)+134B}{n}},$$
and
$$d=\frac{102B\log(\mathcal{N}_n)+272B+3C_nK\big(\log\big(\frac{n}{C_n}\big)+B\big)}{n}+ 3\delta,$$ we get from \eqref{Eq: Pre-epsilon-aid equation after substitution bounds on E[T] and E[T^2]} that $|a-b|\leq 2\sqrt{a}c+d$. Since the excess risk is always nonnegative we can apply Proposition \ref{P: Epsilon aid inequality}. This gives us  for any $0<\epsilon\leq 1$
\begin{equation*}
\begin{aligned}
R_B(\bp_0,\widehat{\bp})
&\leq (1+\epsilon)\left(R_{B,n}(\bp_0,\widehat{\bp})+3\delta+\frac{102B\log(\mathcal{N}_n)+272B+3C_nK\big(\log\big(\frac{n}{C_n}\big)+B\big)}{n}\right)\\
&+\frac{(1+\epsilon)^2}{\epsilon}\cdot\frac{17B\log(\mathcal{N}_n)+136B}{n}.
\end{aligned}
\end{equation*}
Proposition \ref{P: Upper Bound Risk Delta} gives
$R_{B,n}(\bp_0,\widehat{\bp})\leq \inf_{\bp\in\mF}R(\bp_0,\bp)+\Delta_n(\bp_0,\widehat{\bp}).$
Substituting this in the previous equation and observing that $(1+\epsilon)/\epsilon\geq 2$, $1/\epsilon\geq 1$ and $0<1-\epsilon\leq 1$ for $\epsilon\in(0,1]$ yields the assertion of the theorem.
\end{proof}

\bibliography{RefLibMultiClassOverleafVersion}
\bibliographystyle{acm}

\appendix
\section*{Appendix}
\section{Basic network properties and operations}

In this section we state elementary properties of network classes and introduce small networks that are capable of approximating multiplication operations based on similar results in \cite{NonParametricRegressionReLU}. 

\subsubsection{Embedding properties of neural network function classes}
\label{sec.embed_props}

This section extends the results in \cite{NonParametricRegressionReLU} to arbitrary output activation function.

\textit{Enlarging:}
 Let $\bfm$ and $\bfm'$ be two width-vectors of the same length and let $s,s'>0$. If $\bfm\leq \bfm'$ component-wise, $m_{L+1}=m'_{L+1}$ and $s\leq s'$, then 
 \begin{equation}\label{Eq: Enlarging}
 \mF_{\bpsi}(L,\bfm,s)\subseteq \mF_{\bpsi}(L,\bfm',s').\\
 \end{equation}

This rule allows us to simplify the neural network architectures. For example we can simplify a network class by embedding it in a class for which all hidden layers have the same width.

\textit{Composition:}
Let $\bf\in \mF_{\id}(L,\bfm,s_1)$ and let $\bg$ be a network in  $\mF_{\bpsi}(L',\bfm',s_2)$, with $m_{L+1}=m_0'$. For a vector $\mathbf{v}\in\mathbb{R}^{m_{L+1}}$, define the composed network $\bg\circ \sigma_{\mathbf{v}}(\bf)$. Then 
\begin{equation}\label{Eq: Composition}
\bg\circ \sigma_{\mathbf{v}}(\bf) \in\mF_{\bpsi}\big(L+L'+1,(m_0,\cdots,m_{L+1},m_1',\cdots,m_{L'+1}'),s_1+s_2+|\mathbf{v}|_0\big).
 \end{equation}

The following rule allows us to synchronize the depths of neural networks.

\textit{Depth synchronization:}
For any positive integer $a$, 
\begin{equation}\label{Eq: Depth synchronization}
\mF_{\bpsi}(L,\bfm,s)\subset \mF_{\bpsi}(L+a,(\underbrace{m_0,\cdots,m_0}_{a \text{ times}},\bfm),s+am_0).
\end{equation}
To identify simple neural network architectures, we can combine the depth-synchronization and enlarging properties. When there exist $c\geq m_0$ and $b>0,$ such that $s=cL+b,$ and $L^*$ is an upper bound on $L$, combining the previous two properties yields
\begin{equation*}
    \mF_{\bpsi}(L,\bfm,s)\subset \mF_{\bpsi}(L^*,\bfm',cL+m_0(L^*-L)+b)\subset\mF_{\bpsi}(L^*,\bfm',cL^*+b),
\end{equation*}
where the width vector $\bfm'$ has length $L^*+2$ and can be chosen as $(m_0,m',m',\cdots,m',m_{L+1})$ with $m'$ equal to the largest coefficient of $\bfm.$

\textit{Parallelization:}
 Let $\bfm$, $\bfm'$ be two width vectors such that $m_0=m_0'$ and let $\bf\in\mF_{\id}(L,\bfm)$ and $\bg\in \mF_{\id}(L,\bfm')$. Define the parallelized network $\mathbf{h}$ as $\mathbf{h}:=(\bf,\bg)$. Then
\begin{equation}\label{Eq: Parallelization}
\mathbf{h}\in \mF_{\id}(L,(m_0,m_1+m_1',\cdots,m_{L+1}+m_{L+1}').
\end{equation}

\begin{propositie}[Removal of inactive nodes]\label{P: Inactive Node Removal}
 It holds that
 \begin{equation*}
     \mF_{\bpsi}(L,\bfm,s)=\mF_{\bpsi}(L,(m_0,m_1\wedge s,\cdots, m_L\wedge s,m_{L+1}),s).
 \end{equation*}
\end{propositie}
For this property, the output function plays no role and the proof in \cite{NonParametricRegressionReLU} carries over.

The following equation gives the number of parameters in a fully connected network in $\mF_{\bpsi}(L,\bfm)$:
\begin{equation}\label{Eq: Full Parameter Count}
\sum_{j=0}^L(m_j+1)m_{j+1}-m_{L+1}.
\end{equation}
This will be used further on as an upper bound on the number of active parameters in sub-networks.

\subsubsection{Scaling numbers}
We constraint all neural network parameters to be bounded in absolute value by one. To build neural networks with large output values we construct small rescaling networks.

\begin{propositie}\label{P: Large Scale}
For any real number $C$ there exists a network \emph{Scale}$_{C}\in\mF_{\id}(\ceil{\log_2(|C|)}+(\ceil{\log_2(|C|)}-1),(1,2,1,2,1,\cdots,1,2,1),4\ceil{\log_2(|C|)})$ such that \emph{Scale}$_{C}(x)=C(x)_+$.
\end{propositie}
\begin{proof}
Set
$$W_0=\begin{pmatrix}
1\\1
\end{pmatrix},\ \ \ \mathbf{v}_1=\begin{pmatrix}
0 \\ 0
\end{pmatrix},\ \text{ and } \ W_1=(1,1).$$ The network $W_1\sigma_{\mathbf{v}_1}W_0x$ computes $x\mapsto 2(x)_+$. This network has exactly one hidden layer, one input node, one output node and two nodes in the hidden layer. It uses four nonzero-parameters. Composing $\ceil{\log_2(|C|)}$ of these networks, using the composition rule \eqref{Eq: Composition}, where we take the output layer of one network to be the input layer of the next one with shift vector zero, yields a network in the right network class computing $x\mapsto 2^{\ceil{\log_2(|C|)}}(x)_+$.
Replacing the last weight matrix by $(C2^{-\ceil{\log_2(|C|)}},C2^{-\ceil{\log_2(|C|)}})$ yields the result.
\end{proof}

\subsubsection{Negative numbers}
For negative input, the ReLU activation without shift returns zero. As a result, many network constructions output zero for negative input. Using that $x=\sigma(x)-\sigma(-x),$ the next result shows existence of a neural network function that extends the original network function as an even (or odd) function to negative input values.

\begin{propositie}\label{P: Negative numbers.}
 Assume $f\in\mF_{\id}(L,(m_0,m_{1},\cdots,m_{L},1),s)$ and $f(\bx)=0$ whenever $x_j\leq 0$ for some index $j\in\{1,\cdots,m_0\}$. Then there exist neural networks $$f^{\pm}\in\mF_{\id}(L,(m_0,2m_{2},\cdots,2m_{L},1),2s),$$ such that $x_j \mapsto f^+(\bx)$ is an even function, $x_j \mapsto f^-(\bx)$ is an odd function and $f^{\pm}(\bx)=f(\bx)$ for all $\bx$ with $x_j\geq 0$.
\end{propositie}
\begin{proof}
Take two neural networks in the class $\mF_{\id}(L,(m_0,m_{1},\cdots,m_{L},1),s)$ in parallel: The original network $f$ to deal with the positive part and the second network to deal with the negative part. This second network can be build from the first network $f$ by multiplying the $j$-th column vector of $W_0$ by $-1$ and multiplying the output of the network by $\pm1.$ The parallelized network computes then $f^{\pm}.$
\end{proof}

The extension to more than one output is straightforward. Following the same construction as in the previous section, all that has to be done is multiplying the corresponding rows of the weight matrix in the output layer of the neural network by either $-1$, $1$ of $0$ depending on how we wish to extend the function. More precisely, if we have $m_0^-\leq m_0$ input coefficients $x_j$ for which $x_j\leq 0$ implies $f(\bx)=0$, we can find neural networks $$\bf^{\pm}\in\mF_{\id}(L,(m_0,2^{m_0^-}m_{2},\cdots,2^{m_0^-}m_{L},m_{L+1}),2^{m_0^-}s),$$ such that $x_j\mapsto \bf^+(\bx)$ is an even function and $x_j\mapsto \bf^-(\bx)$ is an odd function for all of the $m_0^-$ indices $j.$ This network can be constructed using $2^{m_0^-}$ parallel networks.

\section{Neural networks approximating the logarithm}\label{A:Proof of LogNetwork}
Theorem \ref{T:MSTLogNetworkResult} assumes $M\geq 2$. We use this throughout the proof without further mentioning.
\subsection{Taylor approximation}
Set
\begin{align*}
    T_c^\kappa(x)=\log(c)+ \sum_{\gamma=0}^\kappa x^\gamma \sum_{\alpha=\gamma\vee 1}^{\kappa}\binom{\alpha}{\gamma}\frac{c^{-\gamma}(-1)^{1-\gamma}}{\alpha}=\sum_{\gamma=0}^{\kappa}x^{\gamma}c_{\gamma}.
\end{align*}

\begin{propositie}\label{P: Taylor Bounds For the Logarithm}
For all $\kappa=0,1,\dots$ and every $c>0,$ we have that
\begin{align*}
    \big| \log(x) - T_c^\kappa(x)
    \big|
    \leq \frac{1}{\kappa+1}\left|\frac{x-c}{x\wedge c}\right|^{\kappa+1},
\end{align*}
where the sum in $T_c^{\kappa}$ is defined as zero if $\kappa=0.$ Moreover, if $0<x\leq c,$ we also have that $T_c^\kappa(x)\leq \log(c).$
\end{propositie}
\begin{proof}
We claim that $T_c^{\kappa}$ is equal to the $k$-th order Taylor approximation of the logarithm. First we show that from this claim the statements of the proposition follow. The $\alpha$-th derivative of the logarithm is $\log^{(\alpha)}(x) = (\alpha-1)!(-1)^{\alpha+1} x^{-\alpha}.$ Thus, the $k$-th order Taylor approximation of the logarithm around the point $c$ is given by
\begin{equation}\label{Eq:TaylorLogApprox}
    \log(c)+\sum_{\alpha=1}^{\kappa}\frac{(x-c)^{\alpha}(-1)^{\alpha+1}}{\alpha c^{\alpha}}.
\end{equation}
By the mean value theorem, the remainder is bounded by
\begin{equation*}
\frac{1}{\kappa+1}\left|\frac{x-c}{s}\right|^{\kappa+1},
\end{equation*}
for some $s$ between $x$ and $c.$
Now since the function $1/s$ on $(0,\infty)$ is decreasing, its maximum is obtained at the left boundary, that is, $x\wedge c$, which yields the first claim of the proposition.
Now we show that $T_c^{\kappa}\leq \log(c)$ whenever $0<x\leq c$. When $\kappa=0,$ the sum in \eqref{Eq:TaylorLogApprox} disappears and the result follows immediately. When $\kappa\geq 1$, notice that $(x-c)$ is always negative. Hence the product $(x-c)^{\alpha}(-1)^{\alpha+1}$ is negative for all $\alpha$, so together with the case $\kappa=0$ this yields $T_c^{\kappa}(x)\leq \log(c),$ for $0<x\leq c.$

It remains to prove that $T_c^{\kappa}$ is the $k$-th order Taylor approximation of the logarithm around the point $c$.
Writing the Taylor approximation as a linear combination of monomials gives
\begin{equation*}
\log(c)+\sum_{\alpha=1}^{\kappa}\frac{(x-c)^{\alpha}(-1)^{\alpha+1}}{\alpha c^{\alpha}}=\sum_{\gamma=0}^{\kappa}x^{\gamma}\bar{c}_{\gamma},
\end{equation*}
for suitable coefficients $\bar{c}_{\gamma}$. Using this expression we can obtain the coefficients $\bar{c}_{\gamma}$ for $\gamma\geq 1$ by evaluating the derivatives at $x=0:$
\begin{equation*}
\left.\frac{d^{\gamma}}{dx^{\gamma}} \log(c)+\sum_{\alpha=1}^{\kappa}\frac{(x-c)^{\alpha}(-1)^{\alpha+1}}{\alpha c^{\alpha}}\right|_{x=0}=\gamma!\bar{c}_{\gamma}.
\end{equation*}
This gives us that
\begin{equation*}
\bar{c}_{\gamma}=\sum_{\alpha=\gamma}^{\kappa}\frac{(\alpha-1)!(-c)^{\alpha-\gamma}(-1)^{\alpha+1}}{\gamma!(\alpha-\gamma)!c^{\alpha}}=\sum_{\alpha=\gamma}^{\kappa}\binom{\alpha}{\gamma}\frac{c^{-\gamma}(-1)^{1-\gamma}}{\alpha}.
\end{equation*}
For $\bar{c}_0$ we get
\begin{equation*}
    \bar{c}_0=\log(c)+\sum_{\alpha=1}^{\kappa}\frac{(\alpha-1)!(-c)^{\alpha}(-1)^{\alpha+1}}{(\alpha)!c^{\alpha}}=\log(c)+\sum_{\alpha=1}^{\kappa}\frac{(-1)}{\alpha}.
\end{equation*}
Hence $\sum_{\gamma}^{\kappa}x^{\gamma}\bar{c}_{\gamma}=\sum_{\gamma}^{\kappa}x^{\gamma}c_{\gamma}=T_c^{\kappa}(x),$ proving the claim.
\end{proof}

Next we establish a bound on the sum of the coefficients $c_{\gamma}$ of $T_c^{\kappa}$ in the case $c\leq e$. For $\gamma\geq 1$, we bound $c_{\gamma}$ by 
\begin{equation*}
|c_{\gamma}|\leq  \sum_{\alpha=\gamma}^{\kappa}\binom{\alpha}{\gamma}\frac{(1\wedge c)^{-\gamma}}{\alpha}\leq (1\wedge c)^{-\kappa}\sum_{\alpha=\gamma}^{\kappa}\binom{\alpha}{\gamma}.
\end{equation*}
Since also
\begin{equation*}
    |c_0|\leq |\log(c)|+\sum_{\alpha=1}^{\kappa}\frac{1}{\alpha}\leq |\log(c)|+\sum_{\alpha=1}^{\kappa}\binom{\alpha}{0},
\end{equation*}
this shows that the sum of the coefficients is bounded by
\begin{equation*}
    \sum_{\gamma=0}^{\kappa}|c_{\gamma}|\leq |\log(c)|+(1\wedge c)^{-\kappa}\sum_{\gamma=0}^{\kappa}\sum_{\alpha=1\wedge\gamma}^{\kappa}\binom{\alpha}{\gamma} \leq |\log(c)|+ (1\wedge c)^{-\kappa}\sum_{\gamma=0}^{\kappa}\sum_{\alpha=\gamma}^{\kappa}\binom{\alpha}{\gamma}.
\end{equation*}
The double sum can be rewritten as the sum of all the entries in the rows $0,\cdots,\kappa$ of Pascal's triangle. From the binomial theorem we know that summing over the $\alpha$-th row of Pascal's triangle gives $2^{\alpha}$. Combined with $|\log(c)|\leq (1\wedge c)^{-1}$ for $0<c\leq e$, this gives
\begin{equation}\label{Eq: Bound On Sum of cGamma}
\sum_{\gamma=0}^{\kappa}|c_{\gamma}|\leq (\kappa+1)2^{\kappa+1}(1\wedge c)^{-(\kappa\vee 1)}\leq (\kappa+1)2^{\kappa+1}(1\wedge c)^{-\kappa-1}, \text{ for all } 0<c\leq e.
\end{equation}
Applying the softmax function to an approximation $g$ of the logarithm involves the exponential function and requires a bound for $|e^{g(x)}-x|$ with $x> 0.$ By the mean value theorem $|e^{g(x)}-e^{\log(x)}|=e^{s}|g(x)-\log(x)|$ for a suitable $s$ between $\log(x)$ and $g(x).$ The next proposition provides such a bound. 

\begin{propositie}\label{P: Bound of Exp times Logerror}
For all $\lambda\geq 1$, define
\begin{equation*}
\mathcal{D}_{\lambda}:=\left[\frac{ \lambda^{\ceil{\beta}}}{2^{\ceil{\beta}^2}\ceil{\beta}^{\floor{\beta}}M},\frac{(\lambda+1)^{\ceil{\beta}}}{2^{\ceil{\beta}^2}\ceil{\beta}^{\floor{\beta}}M}\right].
\end{equation*}
If $[a,b]\subset\mathcal{D}_{\lambda}$, then it holds for any $x\in[a,b]$ and any $\omega\leq \log\left(\frac{  (\lambda+1)^{\ceil{\beta}}}{2^{\ceil{\beta}^2}\ceil{\beta}^{\floor{\beta}}M}\right),$ that
\begin{equation*}
e^{\omega}|T_{b}^{\floor{\beta}}(x)-\log(x)|\leq \frac{1}{M}.
\end{equation*}
\end{propositie}
\begin{proof}
First notice that on $(0,\infty)$ the logarithm is strictly increasing and is infinitely times continuously differentiable. For real numbers $a,b$ and a positive integer $j,$ $a^j-b^j=(a-b)\sum_{i=1}^ja^{j-i}b^{i-1}.$ Applied to $a=\lambda+1$ and $b=\lambda,$ this gives $(\lambda+1)^j-\lambda^j\leq j(\lambda+1)^{j-1}$ and thus for $x\in[a,b]\subseteq\mathcal{D}_{\lambda},$ we get that 
\begin{equation*}
|x-b|\leq b-a \leq \frac{ (\lambda+1)^{\ceil{\beta}}-\lambda^{\ceil{\beta}}}{2^{\ceil{\beta}^2}\ceil{\beta}^{\floor{\beta}}M}
\leq b \frac{\ceil{\beta}}{\lambda+1}.
\end{equation*}

Substituting this in the bound from Proposition \ref{P: Taylor Bounds For the Logarithm} and using that $x\geq a$ gives
\begin{equation*}
|T_{b}^{\floor{\beta}}(x)-\log(x)|\leq \frac{1}{\ceil{\beta}}\left|\frac{\ceil{\beta} (\lambda+1)^{\floor{\beta}}}{a2^{\ceil{\beta}^2}\ceil{\beta}^{\floor{\beta}}M}\right|^{\ceil{\beta}}.
\end{equation*}
Since $a\in\mathcal{D}_{\lambda}$,
\begin{equation*}
\begin{aligned}
&|T_{b}^{\floor{\beta}}(x)-\log(x)|\leq \frac{1}{\ceil{\beta}}\left|\frac{\ceil{\beta} (\lambda+1)^{\floor{\beta}}}{2^{\ceil{\beta}^2}\ceil{\beta}^{\floor{\beta}}M}\cdot\frac{2^{\ceil{\beta}^2}\ceil{\beta}^{\floor{\beta}}M}{\lambda^{\ceil{\beta}}}\right|^{\ceil{\beta}}=\ceil{\beta}^{\floor{\beta}}\left|\frac{(\lambda+1)^{\floor{\beta})}}{\lambda^{\ceil{\beta}}}\right|^{\ceil{\beta}}.
\end{aligned}
\end{equation*}

Multiplying both sides with an exponential, noticing that the exponential function is strictly increasing, and applying the upper bound on $\omega$ given in the statement of the proposition yields
\begin{equation*}
\begin{aligned}
e^{\omega}|T_{b}^{\floor{\beta}}(x)-\log(x)|\leq \frac{  (\lambda+1)^{\ceil{\beta}}\ceil{\beta}^{\floor{\beta}}}{2^{\ceil{\beta}^2}\ceil{\beta}^{\floor{\beta}}M}\left|\frac{(\lambda+1)^{\floor{\beta})}}{\lambda^{\ceil{\beta}}}\right|^{\ceil{\beta}}=\frac{1}{2^{\ceil{\beta}^2}M}\left(\frac{\lambda+1}{\lambda}\right)^{\ceil{\beta}^2}.
\end{aligned}
\end{equation*}
Since $(\lambda+1)\lambda^{-1}$ is positive and decreasing for $\lambda\geq 1$, we can upper bound the last display by $1/M$.
\end{proof}

\subsection{Partition of unity}
So far we have bounded the approximation error on subintervals. As we work with ReLU functions, indicator functions of intervals are impractical to use, because they are discontinuous. Instead we create a partition of unity consisting of continuous piecewise linear functions for an interval that contains the interval $[M^{-1},1-M^{-1}]$ .

Define $R$ as the smallest integer sucht that
\begin{equation*}
\frac{(\frac{R}{2}+2^{\ceil{\beta}}\ceil{\beta}^{\floor{\beta}/\ceil{\beta}}-\frac{3}{4})^{\ceil{\beta}}}{2^{\ceil{\beta}^2}\ceil{\beta}^{\floor{\beta}}M}\geq 1-\frac{1}{M}.
\end{equation*}
Rewriting this equation yields
\begin{equation*}
R=\ceil{2^{\ceil{\beta}+1}\ceil{\beta}^{\floor{\beta}/\ceil{\beta}}\left(M-1\right)^{\frac{1}{\ceil{\beta}}}-2\left(2^{\ceil{\beta}}\ceil{\beta}^{\floor{\beta}/\ceil{\beta}}-\frac{3}{4}\right)}\leq2^{\ceil{\beta}+1}\ceil{\beta}^{\floor{\beta}/\ceil{\beta}}M^{\frac{1}{\ceil{\beta}}}.
\end{equation*}

Now we define sequences $(a_r)_{r=1,\cdots,R}$ and $(b_r)_{r=1,\cdots,R-1}$ as follows
\begin{equation*}
a_r:=\frac{(2^{\ceil{\beta}}\ceil{\beta}^{\floor{\beta}/\ceil{\beta}}+\frac{r}{2}-\frac{3}{4})^{\ceil{\beta}}}{2^{\ceil{\beta}^2}\ceil{\beta}^{\floor{\beta}}M},
\end{equation*}
\begin{equation*}
b_r:=\frac{(2^{\ceil{\beta}}\ceil{\beta}^{\floor{\beta}/\ceil{\beta}}+\frac{r}{2}-\frac{1}{2})^{\ceil{\beta}}}{2^{\ceil{\beta}^2}\ceil{\beta}^{\floor{\beta}}M},
\end{equation*}
and for ease of notation define $b_0=a_1$ and $b_R=a_R.$
Notice that $[M^{-1},1-M^{-1}]\subseteq[a_1,a_R]\subseteq[M^{-1},1+M^{-1}]$.

Next we define a family of functions  $(F_r)_{r=2,3,\cdots,R}$ and $(H_r)_{r=1,2,\cdots,R}$ on the interval $[a_1,a_R]$.
For $r=2,\cdots, R$ define the function $F_r$ to be zero outside of the interval $[a_{r-1},a_r]$ and to be a linear interpolation between the value one at the point $b_{r-1}$ and the value zero at the boundaries of this interval. In the same way define for $r=2,\cdots, R-1$ the function $H_r$, but  with support on the interval $[b_{r-1},b_r]$ and with interpolation point $a_r.$ Define $H_1$ to be the linear interpolation between the value one at the point $a_1$ and the value zero at $b_1$ and let it be zero outside this interval. Finally define $H_R$ as the linear interpolation between the value one at the point $b_R$ and the value zero at $b_{R-1}$ and set it to zero outside of this interval.

By construction it holds that 
\begin{equation*}
\sum_{r=2}^RF_r(x)+\sum_{r=1}^RH_r(x)=1, \ \ \text{ for all }x\in[a_1,a_R].
\end{equation*}

Figure \ref{fig:FrHrConstruction} gives the first few functions $F_r$ and $H_r$ in the case that $\beta\in(1,2]$.
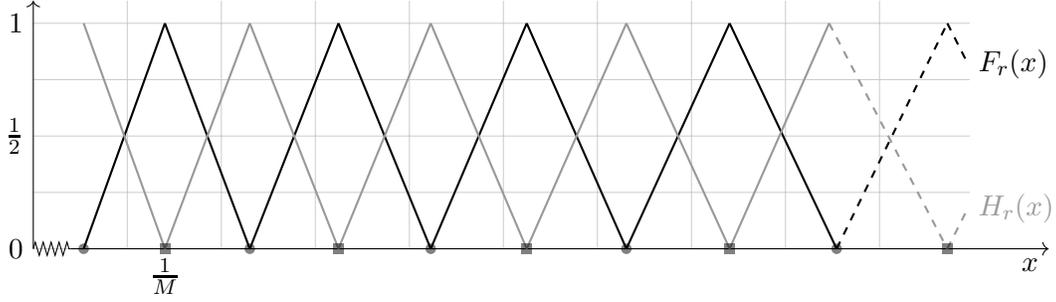
\begin{figure}
    \centering
    \begin{tikzpicture}[decoration={zigzag,segment length=3}]

	\begin{scope}[yscale =3,xscale=5]
	
	\coordinate[InterPoint] (a1) at (0.884,0);
	\coordinate[InterPoint] (a2) at (1.32575,0);
	\coordinate[InterPoint] (a3) at (1.80675,0);
	\coordinate[InterPoint] (a4) at (2.327,0);
	\coordinate[InterPoint] (a5) at (2.886,0);
 	\coordinate[InterPoint] (a6) at (3.48425,0);

	\coordinate[InterPoint] (A1) at (0.884,1);
	\coordinate[InterPoint] (A2) at (1.32575,1);
	\coordinate[InterPoint] (A3) at (1.80675,1);
	\coordinate[InterPoint] (A4) at (2.327,1);
	\coordinate[InterPoint] (A5) at (2.866,1);
 	\coordinate[InterPoint] (A6) at (3.48425,1);

	\coordinate[InterPoint] (b1) at (1.1,0);
	\coordinate[InterPoint] (b2) at (1.5615,0);
	\coordinate[InterPoint] (b3) at (2.062,0);
	\coordinate[InterPoint] (b4) at (2.6015,0);
	\coordinate[InterPoint] (b5) at (3.18025,0);
 	\coordinate[InterPoint] (b6) at (3.798,0);

	\coordinate[InterPoint] (B1) at (1.1,1);
	\coordinate[InterPoint] (B2) at (1.5615,1);
	\coordinate[InterPoint] (B3) at (2.062,1);
	\coordinate[InterPoint] (B4) at (2.6015,1);
	\coordinate[InterPoint] (B5) at (3.18025,1);
 	\coordinate[InterPoint] (B6) at (3.798,1);

	\draw[help lines] (0.75,0) grid (3.24,1.1);
	
	\draw[YAxis] (0.75,0) -- (0.75,1.1);
	\draw[decorate] (0.75,0) -- (0.85,0);
	\draw[XAxis] (0.85,0) -- (3.45,0);
	
	\end{scope}

	\draw[Fr] (a1) -- (B1);
	\draw[Fr] (B1) -- (a2);
	\draw[Fr] (a2) -- (B2);
	\draw[Fr] (B2) -- (a3);
	\draw[Fr] (a3) -- (B3);
	\draw[Fr] (B3) -- (a4);
	\draw[Fr] (a4) -- (B4);
	\draw[Fr] (B4) -- (a5);
	\draw[Fr,dashed] (a5) -- (B5);
	\draw[Fr,dashed] (B5) -- ($(B5)!0.18!(a6)$) coordinate[Fr,label=right:$F_r(x)$] (LF) ;

	\draw[Hr] (A1) -- (b1);
	\draw[Hr] (b1) -- (A2);
	\draw[Hr] (A2) -- (b2);
	\draw[Hr] (b2) -- (A3);
	\draw[Hr] (A3) -- (b3);
	\draw[Hr] (b3) -- (A4);
	\draw[Hr] (A4) -- (b4);
	\draw[Hr] (b4) -- (A5);
	\draw[Hr,dashed] (A5) -- (b5);
	\draw[Hr,dashed] (b5) -- ($(b5)!0.18!(A6)$) coordinate[Hr,label=right:$H_r(x)$] (LH);

	\foreach \apoint in {a1,a2,a3,a4,a5}
		\fill[black, opacity=.5] (\apoint) circle (2pt);
	\foreach  \bpoint in {b1,b2,b3,b4,b5}
		\fill[black, opacity=.5] (\bpoint) \Square{2pt};

	\coordinate[label=left:$1$] (y1) at (3.75,3);
	\coordinate[label=left:$\frac{1}{2}$] (y2) at (3.75,1.5);
	\coordinate[label=left:$0$] (y3) at (3.75,0);
	
	\coordinate[label=below:$\frac{1}{M}$] (N1) at (b1);
	
	\coordinate[label=below:$x$] (Lx) at (17,0);
\end{tikzpicture}
    \caption{The first few functions $F_r(x)$ and $H_r(x)$ when $\beta\in(1,2]$. The points $a_r$ are marked with circles \protect\tikz{\fill[black, opacity=.5] circle (2pt);}, while the points $b_r$ are denoted by squares \protect\tikz{\fill[black, opacity=.5] \Square
{2pt};}. }
    \label{fig:FrHrConstruction}
\end{figure}

We can construct a ReLU network that exactly represents the functions $F_r$ and $H_r$. This construction is a modification of the construction of continuous piecewise linear functions as used in \cite{ErrorBoundsYarotsky}. This modification assures that the parameters are bounded by one.
\begin{propositie}\label{P: ReLU Construction Unity}
 For each function $F_r$ and $H_r$ their exists a network $U_{F_r},U_{H_r}\in \mF_{\id}(3((1+\ceil{\beta})^2+\lfloor\log_2(M\ceil{\beta}^{\floor{\beta}})\rfloor),(1,3,3,\cdots,3,1),8((1+\ceil{\beta})^2+\log_2(M\ceil{\beta}^{\floor{\beta}})))$ such that
 $F_r(x)=U_{F_r}(x)$ and $H_r(x)=U_{H_r}(x)$ for all $x\in[a_1,a_R].$
\end{propositie}
\begin{proof}
The functions $F_r$ and $H_r$, $r=2,\cdots,R$, are piecewise linear functions, consisting of four pieces each. This means that these function can be perfectly represented as a linear combination of three ReLU functions. The interpolation points provide the values of the shift vectors. Writing this out for $F_r$ gives
\begin{equation*}
    F_r(x)=\frac{\sigma(x-a_{r-1})}{b_{r-1}-a_{r-1}}+ \left(\frac{1}{b_{r-1}-a_{r-1}}+\frac{1}{a_r-b_{r-1}}\right)\sigma(x-a_{r-1})+ \frac{\sigma(x-a_{r-1})}{a_r-b_{r-1}}.
\end{equation*}  For $H_r$, $r=2,\cdots, R$ this can be done in a similar way. For $H_1$ and $H_R$ we actually only need one ReLU function. The networks weights in this construction are greater than one. The difference between two consecutive points $a_r$ and $b_r$ can be lower bounded by using that for $x,y\geq0$: $(x+y)^{\ceil{\beta}}-x^{\ceil{\beta}}\geq  y^{\ceil{\beta}}.$ Because of
\begin{equation*}
    \frac{(2^{\ceil{\beta}}\ceil{\beta}^{\floor{\beta}/\ceil{\beta}})^{\ceil{\beta}}}{2^{\ceil{\beta}^2}\ceil{\beta}^{\floor{\beta}}M}-\frac{(2^{\ceil{\beta}}\ceil{\beta}^{\floor{\beta}/\ceil{\beta}}-\frac{1}{4})^{\ceil{\beta}}}{2^{\ceil{\beta}^2}\ceil{\beta}^{\floor{\beta}}M}\geq \frac{(\frac{1}{4})^{\ceil{\beta}}}{2^{\ceil{\beta}^2}\ceil{\beta}^{\floor{\beta}}M},
\end{equation*}
 we can upper bound all the network weights by
\begin{equation}\label{Eq: Upper Bound PartitionUnityNetworksCoefficients}
    2^{1+2\ceil{\beta}+\ceil{\beta}^2}\ceil{\beta}^{\floor{\beta}}M,
\end{equation}
which is the inverse of the lower bound on the smallest difference between two consecutive points multiplied by two.
Dividing the multiplicative constants by this bound and combining \eqref{Eq: Composition} the resulting network with the Scale$_{C}(x)$ network from Proposition \ref{P: Large Scale} with $C$ equal to \eqref{Eq: Upper Bound PartitionUnityNetworksCoefficients} yields a network with the required output and parameters bounded by one. The network class is simplified by using the depth-synchronization \eqref{Eq: Depth synchronization} followed by the enlarging property of neural networks \eqref{Eq: Enlarging}.
\end{proof}

The previous partition yields an approximation $T^{\beta}:[a_1,a_R]\rightarrow \mathbb{R}$ of the logarithm on the entire interval $[a_1,a_R]$ via
\begin{equation}\label{Eq: Complete Log approximation equation}
T^{\beta}(x):=\sum_{r=2}^RF_r(x)T^{\floor{\beta}}_{a_r}(x)+\sum_{r=1}^RH_r(x)T^{\floor{\beta}}_{b_r}(x).
\end{equation}
This function depends on $M$ through the sequence of points $a_r$ and $b_r$.

We can now derive the same type of error bound as in Lemma \ref{P: Bound of Exp times Logerror} for all $x\in[0,1]$. For this, define the projection $\pi:[0,1]\rightarrow [a_1,a_R]$, that maps $x\in[0,1]$ to itself, if it is already in the interval $[a_1,a_R],$ and to the closest boundary point otherwise.

\begin{Lemma}\label{L: Error bound for the Complete Log Approximation}
For all $x\in [0,1],$ we have $|e^{T^{\beta}(\pi(x))}-x|\leq M^{-1}.$
\end{Lemma}

\begin{proof}
First consider $x\in (a_1,a_R]$. By construction there exists a unique $r^*\in\{2,3,\cdots,R\}$ and a unique $\bar{r}\in\{1,\cdots,R\}$ such that $x\in (a_{r^*-1},a_{r^*}],$ and $x\in(b_{\bar{r}-1},b_{\bar{r}}]$. By the mean value theorem and \eqref{Eq: Complete Log approximation equation},
\begin{equation*}
\begin{aligned}
&\left|e^{T^{\beta}(x)}-x\right|\leq e^{\xi}\left|T^{\beta}(x)-\log(x)\right|\\
&=e^{\xi}\left|\sum_{r=2}^RF_r(x)T^{\floor{\beta}}_{a_r}(x)+\sum_{r=1}^RH_r(x)T^{\floor{\beta}}_{b_r}(x)-\log(x)(F_{r^*}(x)+H_{\bar{r}}(x))\right|\\
&\leq F_{r^*}(x)e^{\xi}\left|T_{a_{r^*}}^{\floor{\beta}}(x)-\log(x)\right|+H_{\bar{r}}(x)e^{\xi}\left|T_{b_{\bar{r}}}^{\floor{\beta}}(x)-\log(x)\right|,
\end{aligned}
\end{equation*}
where $\xi$ is some number between $T^{\beta}(x)$ and $\log(x)$.
We now want to apply Proposition \ref{P: Bound of Exp times Logerror}. For this we need to find a $\lambda\geq 1$ such that $[a_{r^*-1},a_{r^*}]\cup [b_{\bar{r}-1},b_{\bar{r}}]\in\mathcal{D}_{\lambda}$ and $\xi\leq \max_{y\in\mathcal{D}_{\lambda}}\log(y),$ with $\mathcal{D}_{\lambda}$ as defined by that proposition.
Because of our choice of the sequences of points $a_r$ and $b_r$,
\begin{equation*}
  \lambda:=\max\Big\{\frac{r^*}2+2^{\ceil{\beta}}\ceil{\beta}^{\floor{\beta}/\ceil{\beta}}-\frac{3}{4},\frac{\bar{r}}2+2^{\ceil{\beta}}\ceil{\beta}^{\floor{\beta}/\ceil{\beta}}-\frac{1}{2}\Big\}-1
\end{equation*}
 satisfies $\lambda\geq1$, since $r^*\geq 2$ and $\bar{r}\geq 1$. Furthermore this choice of $\lambda$
 guarantees that $[a_{r^*-1},a_{r^*}]\cup [b_{\bar{r}-1},b_{\bar{r}}]\subseteq\mathcal{D}_{\lambda}$.
For the bound on $\xi$, notice that $x\in [a_{r^*-1},a_{r^*}]\cup [b_{\bar{r}-1},b_{\bar{r}}]$ and that $T^{\beta}(x)=F_{r^*}(x)T_{a_{r^*}}^{\floor{\beta}}(x)+H_{\bar{r}}(x)T_{b_{\bar{r}}}^{\floor{\beta}}(x)$. Combined with the second statement of Proposition \ref{P: Taylor Bounds For the Logarithm}, that is $T_c^{\kappa}\leq \log(c)$ for $0<c\leq x$, and together with $F_{r^*}(x)+H_{\bar{r}}(x)=1$, this yields $\xi\leq \max\{\log(a_{r^*}),\log(b_{\bar{r}})\}.$
Thus we can apply Proposition \ref{P: Bound of Exp times Logerror} and obtain
\begin{equation*}
F_{r^*}(x)e^{\xi}\left|T_{a_{r^*}}^{\floor{\beta}}(x)-\log(x)\right|+H_{\bar{r}}(x)e^{\xi}\left|T_{b_{\bar{r}}}^{\floor{\beta}}(x)-\log(x)\right|\leq F_{r^*}(x)\frac{1}{M}+H_{\bar{r}}(x)\frac{1}{M}=\frac{1}{M},
\end{equation*}
completing the proof for $x\in[a_1,a_R].$

When $x\in[0,a_1]$, notice that $0<a_1<M^{-1}$ and $T^{\beta}(\pi(x))=T_{b_1}^{\floor{\beta}}(a_1)$. Hence by Proposition \ref{P: Taylor Bounds For the Logarithm} together with $b_1=M^{-1},$ we get that
$T^{\beta}(\pi(x))\leq \log(M^{-1})$ proving that both $x$ and $e^{T^{\beta}(\pi(x))}$ are in $[0,M^{-1}]$. Thus the conclusion also holds for $x\in[0,a_1]$.

For $a_R\geq 1,$ the proof follows from $[0,1]\subseteq([0,a_1]\cup [a_1,a_R])$. Thus it remains to study $a_R<1.$ Consider $x\in[a_R,1]$.
Using that $1-M^{-1}\leq a_R<1$ and that $T^{\beta}(\pi(x))=T_{b_R}^{\floor{\beta}}(a_R)=T_{a_R}^{\floor{\beta}}(a_R)$ yields
$T^{\beta}(\pi(x))=\log(a_R)$. This gives us that both $x$ and $e^{T^{\beta}(\pi(x))}$ are in $[a_R,1]\subset[1-M^{-1},1]$, which immediately yields the required bound.
\end{proof}

\subsubsection{Network Construction}
The following result shows how to approximate multiplications with deep ReLU networks. This is required later to construct neural networks mimicking the Taylor-approximation $T^{\beta}$ considered in the previous section.
\begin{Lemma}[Lemma A.3. of \cite{NonParametricRegressionReLU}]\label{L: Combined Multiplication}
For every $\eta\in\mathbb{N}_{\geq 1}$ and $D\in\mathbb{N}_{\geq 1}$, there exists a network Mult$_{\eta}^D\in\mF_{\id}((\eta+5)\ceil{\log_2(D)},(D,6D,6D,\cdots,6D,1))$, such that Mult$_{\eta}^D\in[0,1]$ and
\begin{equation*}
\left|\text{Mult}_{\eta}^D(x_1,\cdots,x_D)-\prod_{i=1}^Dx_i\right|\leq 3^D2^{-\eta},\ \ \text{ for all }(x_1,\cdots,x_D)\in[0,1]^D.
\end{equation*}
Moreover Mult$_{\eta}^D(x)=0$ if one of the coefficients of $\bx$ is zero.
\end{Lemma}
\begin{Remark}\label{R: Parameter Bound Multiplication}
Using \eqref{Eq: Full Parameter Count} the number of parameters in the neural network Mult$_{\eta}^D$ is bounded by $((\eta+5)\ceil{\log_2(D)}+1)42D^2\leq(\eta+5)126D^2\log_2(D).$
\end{Remark}

We now have all the required ingredients to finish the proof of Theorem \ref{T:MSTLogNetworkResult}:
\begin{proof}[Proof of Theorem \ref{T:MSTLogNetworkResult}]
Since $a_1=\sigma(0\cdot x+a_1),$ the projection $\pi$ can be written in terms of ReLU functions as
\begin{equation*}
    \pi(x)=\max\big(a_1,\min(x,a_R)\big)=
    \sigma(0\cdot x+a_1)+\sigma(x-a_1)-\sigma(x-a_R).
\end{equation*}
For $a_R\leq 1,$ all network parameters are bounded by one and this defines a neural network in $\mF_{\id}(1,(1,3,1),8).$ When $a_R>1,$ we replace $\sigma(x-a_R)$ with $\sigma(x-1)$ as we are only interested in input in the interval $[0,1].$
Having thus obtained a value in the interval $[a_1,a_R]$, we can, for any $r\in\{1,\cdots,R\}$, apply the network $U_{F_r}$ from Proposition \ref{P: ReLU Construction Unity} to it. Using depth synchronization \eqref{Eq: Depth synchronization} and parallelization \eqref{Eq: Parallelization}, we can combine the network $U_{F_r}$ with a parallel network that forwards the input value to obtain a network in the network class
\begin{equation*}
    \mF_{\id}\Big(4\big((1+\ceil{\beta})^2+\log_2(M\ceil{\beta}^{\floor{\beta}})\big),(1,3,1,4,\cdots,4,2),13\big((1+\ceil{\beta})^2+\log_2(M\ceil{\beta}^{\floor{\beta}})\big)\Big),
\end{equation*}
that maps $x\in[0,1]$ to $(F_r(\pi(x)),\pi(x)).$
The next step is to construct a network that approximates $F_r(x)T^{\beta}_{a_r}(x).$ Since $a_r\in[M^{-1},1+M^{-1}],$ \eqref{Eq: Bound On Sum of cGamma} allows us, for $\gamma=1,\cdots,\floor{\beta}$, to use the network Mult$_{\eta}^{\gamma+1}$ with input vector $(F_r(\pi(x)),\pi(x),\cdots,\pi(x))$ to compute approximately the function $F_r(\pi(x))\pi(x)^{\gamma},$ and multiply its output with $c_{\gamma}/\ceil{\beta}2^{\floor{\beta}+1}M^{\ceil{\beta}}$. For each $\gamma\in\{1,\cdots,\floor{\beta}\}$ we have a network that approximately computes the function $x \mapsto F_r(\pi(x))\pi(x)^{\gamma}c_{\gamma}/\ceil{\beta}2^{\floor{\beta}+1}M^{\ceil{\beta}}$. We now consider the network that computes these functions in parallel and combines this with a single shallow hidden node network to approximately compute $F_r(\pi(x))c_{0}/\ceil{\beta}2^{\floor{\beta}+1}M^{\ceil{\beta}}$. Making use of parallelization \eqref{Eq: Parallelization}, depth synchronization \eqref{Eq: Depth synchronization} and Remark \ref{R: Parameter Bound Multiplication}, this yields a network $G_{F_r}\in \mF_{\id}(L^*,(1,6(\ceil{\beta})^2,\cdots,6(\ceil{\beta})^2,1),s^*),$ with
\begin{equation*}
\begin{aligned}
L^*&=4((1+\ceil{\beta})^2+\log_2(M\ceil{\beta}^{\floor{\beta}}))+2(\eta+5)\log_2(\ceil{\beta})\\
s^*&=13((1+\ceil{\beta})^2+\log_2(M\ceil{\beta}^{\floor{\beta}}))+(\eta+5)\log_2(\ceil{\beta})126(\ceil{\beta})^3
\end{aligned}
\end{equation*}
 such that $$\left|G_{F_r}(x)- F_{r}(\pi(x))\sum_{\gamma=0}^{\floor{\beta}}\frac{c_{\gamma}}{\ceil{\beta}2^{\floor{\beta}+1}M^{\ceil{\beta}}}\pi(x)^{\gamma}\right|\leq 3^{\ceil{\beta}}2^{-\eta}.$$ Due to the normalization constant $\ceil{\beta}2^{\floor{\beta}+1}M^{\ceil{\beta}}$ it holds that $G_{F_r}(x)\in[-1,1]$ when $\pi(x)$ is in the support of $F_{r}$. If $\pi(x)$ is outside the support of $F_r$, then Lemma \ref{L: Combined Multiplication} guarantees that $G_{F_r}(x)=0$. Similarly for $F_r$ replaced by $H_r,$ we can construct deep ReLU networks $G_{H_r}$ with the same properties.

Using the $R$ networks $G_{H_r}$ and $R-1$ networks $G_{F_r}$ in parallel together with the observation that each $x$ can be in the support of at most one $F_r$ and one $H_r,$ this yields a deep ReLU network with output $\sum_{r=2}^RG_{F_r}(x)+\sum_{r=1}^RG_{H_r}(x),$ such that
\begin{equation*}
    \left|\sum_{r=2}^RG_{F_r}(x)+\sum_{r=1}^RG_{H_r}(x)-\frac{T^{\beta}(\pi(x))}{\ceil{\beta}2^{\floor{\beta}+1}M^{\ceil{\beta}}}\right|\leq 3^{\ceil{\beta}}2^{-\eta+1}.
\end{equation*}
In the next step we compose the network construction with a scaling network. For this we use the scaling network from Proposition \ref{P: Large Scale} with constant $C=\ceil{\beta}2^{\floor{\beta}+1}M^{\ceil{\beta}}.$ Since the input can be negative we use two of those networks in parallel as described in Proposition \ref{P: Negative numbers.}. This gives us a network
\begin{equation*}
    \widetilde{G}\in\mF_{\id}\bigg(L^*+4\log_2\Big(\ceil{\beta}2^{\floor{\beta}+1}M^{\ceil{\beta}}\Big),\mathbf{m}^*,2Rs^*+16\log_2\Big(\ceil{\beta}2^{\floor{\beta}+1}M^{\ceil{\beta}}\Big)\bigg),
\end{equation*} where $\mathbf{m}^*=(1,12R(\ceil{\beta})^2,\cdots,12R(\ceil{\beta})^2,1),$
 such that
\begin{equation*}
  \left|\widetilde{G}(x)-T^{\beta}(\pi(x))\right| \leq \ceil{\beta}2^{\floor{\beta}+2}M^{\ceil{\beta}}3^{\ceil{\beta}}2^{-\eta}.
\end{equation*}
Setting $\eta=\ceil{\log_2(\ceil{\beta}2^{\floor{\beta}+2}M^{\ceil{\beta}+1}3^{\ceil{\beta}})},$ this is upper bounded by $M^{-1}.$
Applying the triangle inequality, the mean value theorem and Lemma \ref{L: Error bound for the Complete Log Approximation} yields
\begin{equation}\label{Eq: Non-truncatedGnetworkBound}
    \left|e^{\widetilde{G}(x)}-x\right|\leq\left|e^{\widetilde{G}(x)}-e^{T^{\beta}\log(\pi(x))}\right|+\left|e^{T^{\beta}\log(\pi(x))}-x\right|\leq \frac{e^{2/M}}{M}+\frac{1}{M}\leq \frac{4}{M},
\end{equation}
where the term $e^{2/M}$ comes from noticing that $|\widetilde{G}(x)-T^{\beta}\log(\pi(x))|\leq M^{-1}$, $|T^{\beta}\log(\pi(x))-\log(1)|\leq M^{-1}$ and triangle inequality.

To derive the lower bound $G(x)\geq \log(4/M),$ we construct a network that computes the maximum between $\widetilde{G}(x)$ and $\log(4/M).$ Since $M\geq 1$ implies $|\log(4/M)|/ \ceil{\beta}2^{\floor{\beta}+1}M^{\ceil{\beta}}\leq 1,$ we can achieve this by adding one additional layer before the scaling. This layer can be written as
\begin{equation}\label{Eq: Lower bound enforcement}
    \sigma\Big(x-\frac{\log(4/M)}{\ceil{\beta}2^{\floor{\beta}+1}M^{\ceil{\beta}}}\Big)+\frac{\log(4/M)}{\ceil{\beta}2^{\floor{\beta}+1}M^{\ceil{\beta}}}\sigma(1).
\end{equation}
 Applying the scaling as before yields a network $G(x)=\max\{\widetilde{G}(x),\log(4/M)\}$ that is in the same network class as $\widetilde{G}(x)$. For the upper bound notice that if $G(x)=\widetilde{G}(x)$, then the bound follows from \eqref{Eq: Non-truncatedGnetworkBound}. When $G(x)=\log(4/M),$ then $\widetilde{G}(x)\leq \log(4/M),$ so \eqref{Eq: Non-truncatedGnetworkBound} implies that $x\leq 8/M.$ Hence
\begin{equation*}
    \left|e^{G(x)}-x\right|=\left|\frac{4}{M}-x\right|\leq \frac{4}{M}.
\end{equation*}
The network size as given in the theorem is an upper bound on the network size obtained here, which is allowed by the depth-synchronization followed by the enlarging property, and is done in order to simplify the expressions.
\end{proof}
Figure \ref{fig:LogNetworkConstruction} shows the main substructures of the deep ReLU network construction in this proof. 

\begin{figure}
\tikzset{Output/.style={}}
\tikzset{Scaling/.style={rectangle,draw=black!60, minimum height=6mm, minimum width=1cm}}
\tikzset{Restriction/.style={}}
\tikzset{NetworkPic/.style={draw=black!60, minimum height=6mm, minimum width=1cm}}
\tikzset{SubFH/.style={thick, black}}
\tikzset{SubMult/.style={thick, black}}
\tikzset{Multiplication/.style={}}
\tikzset{PartNet/.pic={\draw (-0.5cm,-0.15cm)-- (-0.20cm,-0.15cm) -- (0cm,0.25cm) -- (0.30cm,-0.15cm)--(0.5cm,-0.15cm);\draw[NetworkPic] (-0.5cm,-0.2cm) rectangle (0.5cm,0.4cm);}}
\tikzset{PartNetH1/.pic={\draw (-0.5cm,0.25cm)-- (-0.3cm,-0.15cm) -- (0.5cm,-0.15cm);\draw[NetworkPic] (-0.5cm,-0.2cm) rectangle (0.5cm,0.4cm);}}
\tikzset{PartNetHR/.pic={\draw (-0.5cm,-0.15cm) -- (0.20cm,-0.15cm)--(0.5cm,0.25cm);\draw[NetworkPic] (-0.5cm,-0.2cm) rectangle (0.5cm,0.4cm);}}
\tikzset{RestPhi/.pic={\draw (-0.75cm,0cm) -- (-0.5cm,0cm)--(0.5cm,0.5cm)--(0.75cm,0.5cm);\draw[NetworkPic] (-0.75cm,-0.2cm) rectangle (0.75cm,0.6cm);\node at ($(-0.5cm,0cm)!.5!(0.5cm,0.5cm)$) [label={[xshift=0.5cm, yshift=-0.6cm]$\pi(x)$}] {};}}

\tikzset{RestLB/.pic={\draw (-0.75cm,0cm) -- (-0.5cm,0cm)--(0.75cm,0.6cm);\draw[NetworkPic] (-0.75cm,-0.2cm) rectangle (0.75cm,0.6cm);}}
    \centering
    \begin{adjustbox}{max width=\textwidth}
    \begin{tikzpicture}

\pic [local bounding box=NF2] at(-10cm,0cm) {PartNet};
\node (SF2N) [Scaling, right=of NF2] {scale};
\draw[->] (NF2) -- (SF2N) node[midway, above=8pt] (f2N){$U_{F_2}$};
\node[fit=(SF2N)(NF2)(f2N),SubFH,draw] (F2) {};

\pic [local bounding box=NF3, below=1.25cm of NF2] {PartNet};
\node (SF3N) [Scaling, right=of NF3] {scale};
\draw[->] (NF3) -- (SF3N) node[midway, above=8pt] (f3N){$U_{F_3}$};
\node[fit=(SF3N)(NF3)(f3N),SubFH,draw] (F3){};

\pic [local bounding box=NFR, below=2cm of NF3] {PartNet};
\node (SFRN) [Scaling, right=of NFR] {scale};
\draw[->] (NFR) -- (SFRN) node[midway, above=8pt] (fRN){$U_{F_R}$};
\node[fit=(SFRN)(NFR)(fRN),SubFH,draw] (FR) {};

\node at ($(F3.south)!.4!(FR.north)$) {\vdots};

\pic [local bounding box=NH1, below=1.25cm of NFR] {PartNetH1};
\node (SH1N) [Scaling, right=of NH1] {scale};
\draw[->] (NH1) -- (SH1N) node[midway, above=8pt] (h1N){$U_{H_1}$};
\node[fit=(SH1N)(NH1)(h1N),SubFH,draw] {};

\pic [local bounding box=NH2, below=1.25cm of NH1] {PartNet};
\node (SH2N) [Scaling, right=of NH2] {scale};
\draw[->] (NH2) -- (SH2N) node[midway, above=8pt] (h2N){$U_{H_2}$};
\node[fit=(SH2N)(NH2)(h2N),SubFH,draw] (H2) {};

\pic [local bounding box=NHR1, below=2cm of NH2] {PartNet};
\node (SHR1N) [Scaling, right=of NHR1] {scale};
\draw[->] (NHR1) -- (SHR1N) node[midway, above=8pt] (hR1N){$U_{H_{R-1}}$};
\node[fit=(SHR1N)(NHR1)(hR1N),SubFH,draw] (HR1) {};

\node at ($(H2.south)!.4!(HR1.north)$) {\vdots};

\pic [local bounding box=NHR, below=1.25cm of NHR1] {PartNetHR};
\node (SHRN) [Scaling, right=of NHR] {scale};
\draw[->] (NHR) -- (SHRN) node[midway, above=8pt] (hRN){$U_{H_R}$};
\node[fit=(SHRN)(NHR)(hRN),SubFH,draw] (HR) {};

\node[fit={($(HR.south east)+(3pt,-1pt)$)($(F2.north west)+(-3pt,1pt)$)},label=above:Partition of unity,draw,very thick] (partition) {};

\node[below=1.25 cm of hRN,draw, black, very thick,minimum height=10mm, minimum width = 37mm] (Identity) {Identity};

\pic [local bounding box=RESP, left=2.5cm of NH1] {RestPhi};
\node[fit={($(RESP.north west)+(-3pt,3pt)$)($(RESP.south east)+(3pt,-3pt)$)}, label=above:Projection,draw,very thick] (RestrictionPhi) {};

\draw[->] ($(RESP.east)+(-2pt,0)$) -- (NHR.west);
\draw[->] ($(RESP.east)+(-2pt,0)$) -- (NHR1.west);
\draw[->] ($(RESP.east)+(-2pt,0)$) -- (NH2.west);
\draw[->] ($(RESP.east)+(-2pt,0)$) -- (NH1.west);
\draw[->] ($(RESP.east)+(-2pt,0)$) -- (NFR.west);
\draw[->] ($(RESP.east)+(-2pt,0)$) -- (NF3.west);
\draw[->] ($(RESP.east)+(-2pt,0)$) -- (NF2.west);
\draw[->] ($(RESP.east)+(-2pt,0)$) -- ($(NH2.south west)!.4!(NHR1.north west)$);
\draw[->] ($(RESP.east)+(-2pt,0)$) -- ($(NF3.south west)!.4!(NFR.north west)$);

\draw[->] ($(RESP.east)+(-2pt,0)$) --(Identity.west);

\node [left=1cm of RESP,label=above:Input] (Input) {$x$};
\draw[->] (Input) -- (RESP);

\node[Multiplication, above right= 0cm and 7cm  of NF2] (Mult1) {Identity};
\node[Multiplication, below=0.1mm of Mult1] (Mult1for) {$F_2(\pi(x))$};
\node[fit={($(Mult1for.south east)+(0.55cm,0cm)$)($(Mult1for.south west)+(-0.55cm,0cm)$)(Mult1)},SubMult,draw] (M1) {};

\node[Multiplication, below=1cm of Mult1] (Mult2) {Mult$_{\eta}^{2}$};
\node[Multiplication, below=0.1mm of Mult2] (Mult2for) {$F_2(\pi(x))\pi(x)$};
\node[fit={($(Mult2for.south east)+(0.22cm,0cm)$)($(Mult2for.south west)+(-0.22cm,0cm)$)(Mult2)},SubMult,draw] (M2) {};

\node[Multiplication, below=1.5cm of Mult2] (MultB) {Mult$_{\eta}^{\floor{\beta}}$};
\node[Multiplication, below=0.1mm of MultB] (MultBfor) {$F_2(\pi(x))\pi(x)^{\floor{\beta}}$};
\node[fit={(MultB)(MultBfor)},SubMult,draw] (MB) {};

\node at ($(Mult2for.south)!.4!(MultB.north)$) {\vdots};

\node[fit={($(M1.north west)+(-3pt,10pt)$)(M2)($(MB.south east)+(3pt,0pt)$)},draw, very thick] (GF2) {};
\node[below] at (GF2.north) {$G_{F_2}$};

\draw[->] (SF2N.east) -- (M1.west);

\draw[->] (SF2N.east) -- (M2.west);
\draw[->] (Identity.east) -- (M2.west);
\draw[->] (SF2N.east) -- ($(M2.west)!.5!(MB.west)$);
\draw[->] (Identity.east) -- ($(M2.west)!.5!(MB.west)$);
\draw[->] (SF2N.east) -- (MB.west);
\draw[->] (Identity.east) -- (MB.west);

\node[below=0.1cm of GF2, draw, very thick, minimum width =34mm, minimum height=8mm] (GF3) {$G_{F_3}$};

\node[below=1cm of GF3, draw, very thick, minimum width =34mm, minimum height=8mm] (GFR) {$G_{F_R}$};
\node[below=0.1cm of GFR, draw, very thick, minimum width =34mm, minimum height=8mm] (GH1) {$G_{H_1}$};
\node[below=0.1cm of GH1, draw, very thick, minimum width =34mm, minimum height=8mm] (GH2) {$G_{H_2}$};

\node[below=1cm of GH2, draw, very thick, minimum width =34mm, minimum height=8mm] (GHR1) {$G_{H_{R-1}}$};
\node[below=0.1cm of GHR1, draw, very thick, minimum width =34mm, minimum height=8mm] (GHR) {$G_{H_{R}}$};

\node at ($(GF3.south)!.4!(GFR.north)$) {\vdots};
\node at ($(GH2.south)!.4!(GHR1.north)$) {\vdots};

\draw[->] (SF3N.east) -- (GF3.west);
\draw[->,black!50] (Identity.east) -- (GF3.west);

\draw[->] ($(SF3N.south east)!.4!(SFRN.north east)$) -- ($(GF3.south west)!.4!(GFR.north west)$);
\draw[->,black!50] (Identity.east) -- ($(GF3.south west)!.4!(GFR.north west)$);

\draw[->] (SFRN.east) -- (GFR.west);
\draw[->,black!50] (Identity.east) -- (GFR.west);

\draw[->] (SH1N.east) -- (GH1.west);
\draw[->,black!50] (Identity.east) -- (GH1.west);

\draw[->] (SH2N.east) -- (GH2.west);
\draw[->,black!50] (Identity.east) -- (GH2.west);

\draw[->] ($(SH2N.south east)!.4!(SHR1N.north east)$) -- ($(GH2.south west)!.4!(GHR1.north west)$);
\draw[->,black!50] (Identity.east) -- ($(GH2.south west)!.4!(GHR1.north west)$);

\draw[->] (SHR1N.east) -- (GHR1.west);
\draw[->,black!50](Identity.east) -- (GHR1.west);

\draw[->] (SHRN.east) -- (GHR.west);
\draw[->,black!50] (Identity.east) -- (GHR.west);

\pic [local bounding box=RESPLB, right=3.5cm of GF3] {RestLB};
\node[fit={($(RESPLB.north west)+(-3pt,3pt)$)($(RESPLB.south east)+(3pt,-3pt)$)}, label={[align=center]above:\eqref{Eq: Lower bound enforcement}},draw,very thick] (RestrictionLB) {};

\draw[->] (M1.east) -- (RESPLB.west) node[above, near start] {$c_0$};
\draw[->] (M2.east) -- (RESPLB.west) node[above, near start] {$c_1$};
\draw[->] ($(M2.east)!.5!(MB.east)$) -- (RESPLB.west) node[above, near start] {$c_{\gamma}$};
\draw[->] (MB.east) -- (RESPLB.west) node[above,near start] {$c_{\floor{\beta}}$};

\draw[->] (GF3.east) -- (RESPLB.west) node[above,near start] {$\mathbf{c}$};
\draw[->] ($(GF3.south east)!.4!(GFR.north east)$) -- (RESPLB.west);
\draw[->] (GFR.east) -- (RESPLB.west) node[above,near start] {$\mathbf{c}$};
\draw[->] (GH1.east) -- (RESPLB.west) node[above,near start] {$\mathbf{c}$};
\draw[->] (GH2.east) -- (RESPLB.west) node[above,near start] {$\mathbf{c}$};
\draw[->] ($(GH2.south east)!.4!(GHR1.north east)$)-- (RESPLB.west);
\draw[->] (GHR1.east) -- (RESPLB.west) node[above,near start] {$\mathbf{c}$};
\draw[->] (GHR.east) -- (RESPLB.west) node[below right,near start] {$\mathbf{c}$};

\node (SP) [Scaling, above right=of RESPLB, label={[align=center]above:positive\\part}] {scale};
\node (SN) [Scaling, below right=of RESPLB, label={[align=center]above:negative\\part}] {scale};

\draw[->] (RESPLB.east) -- (SP.west) node[above, midway] {$+$};
\draw[->] (RESPLB.east) -- (SN.west) node[above, midway] {$-$};

\node [below right=1cm and 1.2cm of SP,label=above:Output] (Output) {$G(x)$};
\draw[->] (SP.east) -- (Output.west) node[above, near start] {$+$};
\draw[->] (SN.east) -- (Output.west) node[above, near start] {$-$};

\node[fit={(SP)($(SN.east)+(0.6cm,0cm)$)(GF2)(GHR)($(GF3.west)+(-0.5cm,0cm)$)}, draw, very thick,label=above:Taylor] (Taylor) {};

\node[fit={($(Taylor.north east)+(0cm,0.5cm)$)(RestrictionPhi)(Identity)(partition)}, draw, ultra thick, label=above:$G(x)$] (TGX) {};

\end{tikzpicture}
    \end{adjustbox}
    \caption{The construction of the logarithm approximation network $G$ of Theorem \ref{T:MSTLogNetworkResult} from subnetworks. The difference between the networks $G$ and $\widetilde{G}$ is the single layer which enforces the lower bound, which is not present in the network $\widetilde{G}$.}
    \label{fig:LogNetworkConstruction}
\end{figure}
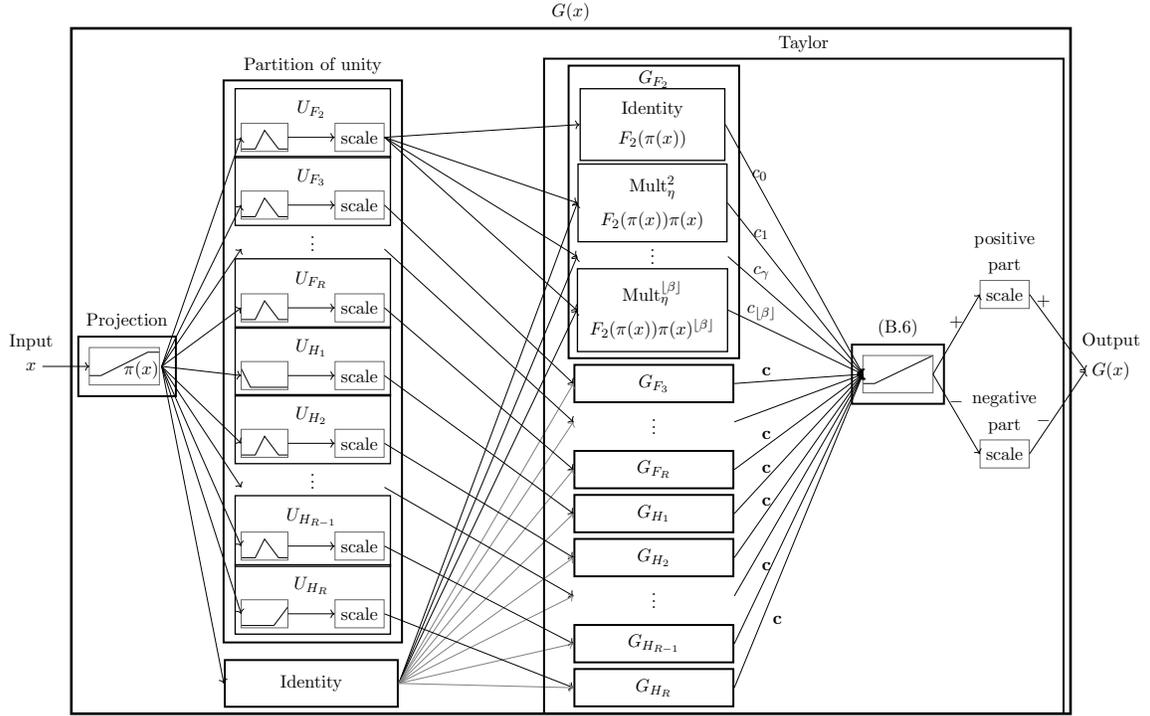

\section{Further technicalities}
\label{appendix.proofs}

\begin{propositie}[Bernstein's inequality]\label{P: Bernsteins Inequality}
For independent random variables $(Z_i)_{i=1}^n$ with zero mean and moment bounds $\mathbb{E}|Z_i|^m\leq \tfrac{1}{2}m!U^{m-2}v_i$ for $m=2,3,\dots$ and $i=1,\dots,n$ for some constants $U$ and $v_i$, we have
$$\mathbb{P}\left(\left|\sum_{i=1}^nZ_i\right|> x\right)\leq 2e^{-\frac{x^2}{2v+2Ux}}, \quad \text{for}  \ v\geq \sum_{i=1}^nv_i.$$
\end{propositie}
This formulation of Bernstein's inequality is based on the formulation in Lemma 2.2.11 of \cite{WeakConvergenceEmpProcesses}. The proof can be found in \cite{bennett1962probability}.

The next elementary inequality generalizes Lemma 10 of \cite{NonParametricRegressionReLU}.
\begin{Lemma}\label{P: Epsilon aid inequality}
If $a,b,c,d$ are real numbers, $a\geq 0$, such that $|a-b|\leq 2\sqrt{a}c+d$, then, for each $\epsilon\in(0,1]$,
\begin{equation*}
(1-\epsilon)(b-d)-\frac{(1-\epsilon)^2}{\epsilon}c^2\leq a \leq (1+\epsilon)(b+d)+\frac{(1+\epsilon)^2}{\epsilon}c^2.
\end{equation*}
\end{Lemma}
\begin{proof}
First notice that $|a-b|\leq 2\sqrt{a}c+d$ if and only if $-2\sqrt{a}c-d\leq a-b\leq 2\sqrt{a}c+d.$
Using that
$2xy\leq x^2+y^2$ for all $x,y\in\mathbb{R}$, we get for $x:=\sqrt{a}\sqrt{\epsilon}/\sqrt{1+\epsilon}$ and $y:=c\sqrt{1+\epsilon}/\sqrt{\epsilon}$, that
$$2\sqrt{a}c=2xy\leq x^2+y^2=\frac{\epsilon a}{1+\epsilon}+\frac{(1+\epsilon)c^2}{\epsilon}$$
and therefore $$a-b\leq \frac{\epsilon a}{1+\epsilon}+\frac{(1+\epsilon)c^2}{\epsilon}+d.$$ Rearranging the terms yields the upper bound of the lemma.
For the lower bound notice that if $\epsilon=1$, then the lower bound is zero, and holds since $a\geq 0$. For $\epsilon\in(0,1)$ using the same argument but now with $x=\sqrt{a}\sqrt{\epsilon}/\sqrt{1-\epsilon}$ and $y=c\sqrt{1-\epsilon}/\sqrt{\epsilon}$, gives
$$a-b\geq -\frac{\epsilon a}{1-\epsilon}-\frac{(1-\epsilon)c^2}{\epsilon}-d.$$ Rearranging the terms yields the lower bound of the proposition.
\end{proof}
The number $a$ is required to be nonnegative as otherwise $\sqrt{a}$ would not be a real number.
In the statement in \cite{NonParametricRegressionReLU} the constants $a,b,c,d$ are all required to be positive. However since the inequality $2xy\leq x^2+y^2$ holds for all real numbers $x,y$ the positivity constraint is not necessary. However, when $c$ and $d$ are negative the term $2\sqrt{a}c+d$ is negative, and no pair $a,b$ exists such that the condition is satisfied.

Recall that $d_{\tau}(\bf,\bg):=\sup_{\bx\in\mathcal{D}}\max_{k=1,\cdots,K}|(\tau\vee f_k(\bx))-(\tau\vee g_k(\bx))|.$ Observe that $d_{\tau}(\bf,\bg)=0$ does not imply $\bf=\bg$, which is why $d_{\tau}$ is not a metric.
The next lemma shows that this, however, defines a pseudometric.
\begin{Lemma}\label{P: truncation log pseudometric}
Let $\bf,\bg,\mathbf{h}:\mathcal{D}\rightarrow \mathbb{R}^K$, then for every $\tau\in\mathbb{R}$:
\begin{compactitem}
\item[(i)] $d_{\tau}(\bf,\bg)\geq 0$
\item[(ii)] $d_{\tau}(\bf,\bf)=0$
\item[(iii)] $d_{\tau}(\bf,\bg)=d_{\tau}(\bg,\bf)$
\item[(iv)] $d_{\tau}(\bf,\bg)\leq d_{\tau}(\bf,\mathbf{h})+d_{\tau}(\mathbf{h},\bg).$
\end{compactitem}

\end{Lemma}
\begin{proof}
(i), (ii) and (iii) follow immediately. (iv) follows from applying triangle inequality to the $\|\cdot\|_{\infty}$ norm,
\begin{equation*}
\begin{aligned}
d_{\tau}(\bf,\bg)&=\big\|\max_{k=1,\cdots,K}|(\tau\vee f_k(\cdot))-(\tau\vee g_k(\cdot))|\big\|_{\infty}\\
&\leq\big\|\max_{k=1,\cdots,K}|(\tau\vee f_k(\cdot))-(\tau\vee h_k(\cdot))|\big\|_{\infty}+\big\|\max_{k=1,\cdots,K}|(\tau\vee h_k(\cdot))-(\tau\vee g_k(\cdot))|\big\|_{\infty}\\
&=d_{\tau}(\bf,\mathbf{h})+d_{\tau}(\mathbf{h},\bg).
\end{aligned}
\end{equation*}
\end{proof}

\begin{Lemma}
If $\mG$ is a function class of functions from $\mathcal{D}$ to $[0,\infty)^K$, then for all $\delta>0$ and $\tau>0$
$$\mathcal{N}\big(\delta,\log(\mG),d_{\log(\tau)}(\cdot,\cdot)\big)\leq\mathcal{N}\big(\delta \tau,\mG,d_{\tau}(\cdot,\cdot)\big).$$
\end{Lemma}
\begin{proof}
Let $\delta>0$. Denote by $(\bg_j)_{j=1}^{\mathcal{N}_n}$ the centers of a minimal internal $\delta \tau$-covering of $\mG$ with respect to $d_{\tau}$ and let $\bg\in\mG$. By the cover property, there exist a $j\in\{1,\cdots,\mathcal{N}_n\}$ such that
$d_{\tau}(\bg,\bg_j)\leq \delta \tau$.

The derivative of $\log(u)$ is $1/u$, so the logarithm is Lipschitz on $[\tau,\infty)$ with Lipschitz constant $\tau^{-1}$. Applying this to $d_{\log(\tau)}(\log(\bg),\log(\bg_j))$, noticing that $\max\{\log(\tau),\log(x)\}\in[\log(\tau),\infty)$ for $x\in[0,\infty)$, yields
\begin{equation*}
\begin{aligned}
&\max_{\bx\in\mathcal{D}}\max_{k=1,\cdots,K}|(\log(\tau)\vee \log(g_k(\bx)))-(\log(\tau)\vee \log(g_{j,k}(\bx)))|\\
&\leq \tau^{-1} \max_{\bx\in\mathcal{D}}\max_{k=1,\cdots,K}|(\tau\vee g_k(\bx))-(\tau\vee g_{j,k})(\bx))|\\
&\leq \tau^{-1}\delta \tau=\delta.
\end{aligned}
\end{equation*}
Since $\bg\in\mG$ was arbitrary, this means that for all $\bg\in\mG$ there exists a $j\in \{1,\cdots,\mathcal{N}_n\}$ such that
$d_{\log(\tau)}(\log(\bg),\log(\bg_j))\leq \delta.$ Hence $(\log(\bg_j))_{j=1}^{\mathcal{N}_n}$ is a $\delta$-cover for $\log(\mG)$ with respect to $d_{\log(\tau)}$. Since the $\bg_j$ are in $\mG$, the $\log(\bg_j)$ are in $\log(\mG)$, thus this cover is an internal cover. Since $\mathcal{N}(\delta,\log(\mG),d_{\log(\tau)}(\cdot,\cdot))$ is the minimal number of balls with center in $\log(\mG)$ required to cover $\log(\mG)$. This proves the assertion.
\end{proof}

\begin{proof}[Proof of Lemma \ref{L: Inequality to relate to the risk}]
Let $\mathbf{p},\mathbf{q}\in\mathcal{S}^k$. Thus, $\sum_{k=1}^Kp_k=1,$  $\sum_{k=1}^Kq_k=1$ and
\begin{equation}\label{Eq: Inequality Adding 1 Trick}
\sum_{k=1}^Kp_k\left(B\wedge\log\left(\frac{p_k}{q_k}\right)\right)=\sum_{k=1}^K\left(p_k\left(B\wedge\log\left(\frac{p_k}{q_k}\right)\right)-p_k+q_k\right).
\end{equation}

Suppose for the moment that for any $k=1,\cdots, K,$
\begin{equation}\label{Eq: Termwise Objective}
p_k\left(B\wedge\log\left(\frac{p_k}{q_k}\right)\right)-p_k+q_k\geq \frac{1}{C_{m,B}}p_k\left|B\wedge\log\left(\frac{p_k}{q_k}\right)\right|^m,
\end{equation}
with $C_{m,B}:=\max\{m!,B^m/(B-1)\}.$ Applying this inequality to each term on the right hand side of \eqref{Eq: Inequality Adding 1 Trick} gives
\begin{equation*}
\sum_{k=1}^Kp_k\left(B\wedge\log\left(\frac{p_k}{q_k}\right)\right)\geq \sum_{k=1}^K \frac{1}{C_{m,B}}p_k\left|B\wedge\log\left(\frac{p_k}{q_k}\right)\right|^m.
\end{equation*}
Since $C_{m,B}>0,$ multiplying both sides of the inequality with $C_{m,B}$ yields the claim.

It remains to proof \eqref{Eq: Termwise Objective}. First we consider the case that $p_k=0$. By considering the limit we get that $0\log^m(0)=0$, for $m=1,2,\cdots$. Thus the right hand side of \eqref{Eq: Termwise Objective} is equal to $0$, while the left hand side is equal to $q_k$. Since $q_k\geq 0$, this proves \eqref{Eq: Termwise Objective} for this case.

Assume now that $p_k>0$. Dividing both sides by $p_k$ yields
$$B\wedge\log\left(\frac{p_k}{q_k}\right)-1+\frac{q_k}{p_k}\geq \frac{1}{C_{m,B}}\left|B\wedge\log\left(\frac{p_k}{q_k}\right)\right|^m.$$
If $p_k/q_k\geq e^B$ the inequality follows immediately. It remains to study the case that $p_k/q_k<e^B.$ In this case one can always replace $B\wedge \log(p_k/q_k)$ by $\log(p_k/q_k)$.
Introducing the new variable $u=q_k/p_k$ and replacing $C_{m,B}$ by $C>0$ gives rise to a function
$$H_{C,m}(u)=u-1-\log(u)-|\log(u)|^m/C.$$
It remains to show that $H_{C_{m,B},m}(u)\geq 0$ for all $u\geq e^{-B}$.
Obviously, $H_{C,m}(1)=0$ for all $C$, so we only have to consider $u\neq 1$.
Consider first $u>1$ and $C=m!.$ Using the substitution $u=e^s$ gives
$$m!e^s-m!(s+1)-s^m.$$
Substituting the power series for the exponential function leads to
\begin{equation*}
m!\sum_{n=0}^{\infty}\frac{s^n}{n!} -m!(1+s)-s^m=m!\sum_{n=2}^{m-1}\frac{s^n}{n!}+m!\sum_{n=m+1}^{\infty}\frac{s^n}{n!}>0,
\end{equation*}
where the last strict inequality holds because $u>1$ and thus $s>0$. Thus $H_{m!,m}(u)\geq 0$ for $u>1$.

For $u\in(e^{-b},1)$, dividing by $u-\log(u)-1$ gives us the following constraint on the constant $C:$
\begin{equation}\label{Eq: The equation for the constant}
C\geq \sup_{u\in(e^{-B},1)}\frac{|\log(u)|^m}{u-\log(u)-1}.
\end{equation} This division can be done since $u-\log(u)-1>0$ when $u>0$, $u\neq 1$ and zero if and only if $u=1$, which for example can be shown by observing the sign of the derivative. 

Define $C_{<1}$ as $C_{<1}:=B^m/(B-1).$ Since $|\log(u)|^m/(u-\log(u)-1)$ is strictly decreasing on $(0,1)$, see Proposition \ref{P: Collection of Existence and Uniqueness results} (II), it follows for $u\in[e^{-B},1)$ that $|\log(u)|^m/(u-\log(u)-1)\leq B^m/(e^{-B}+B-1).$ Now since $B>1$, it follows that $B^m/(u+B-1)$ is also strictly decreasing on $[0,1]$. Hence on $[0,e^{-B}]$ we have $B^m/(e^{-B}+B-1)\leq B^m/(u+B-1)\leq C_{<1}$, thus $C_{<1}$ satisfies \eqref{Eq: The equation for the constant}.

Now notice that $C_{m,B}=\max\{C_{<1},m!\}.$ Consequently $H_{C_{m,B},m}(u)\geq 0,$ for all $u\geq e^{-B},$ proving \eqref{Eq: Termwise Objective}.
\end{proof}

For all $m=2,3,\dots$ define the function $F_m:(0,\infty)\rightarrow [0,\infty)$ as
\begin{equation*}
F_m(u):=\frac{|\log^m(u)|}{u-\log(u)-1}.
\end{equation*}
Since $u-\log(u)-1\geq 0$, this function indeed takes only positive values. Furthermore since $u-\log(u)-1=0$ only when $u=1$ this is the only possible singularity/discontinuity of this function. The next result derives some properties of the function $F_m(u).$
\begin{propositie}\label{P: Collection of Existence and Uniqueness results}
If $m=2,3,\cdots,$ then
\begin{compactitem}
\item[(i)] $\lim_{u\rightarrow 1}F_2(u)= 2$ and $\lim_{u\rightarrow 1}F_m(u)= 0$ for $m >2$
\item[(ii)] $F_{m}(u)$ is strictly decreasing on $(0,1)$.
\end{compactitem}

\end{propositie}
\begin{proof} (i):\ For $u=1$, it holds that $(u-\log(u)-1)=0$ and $|\log^m(u)|=0$. Applying L'Hopital's rule twice yields the desired result.\ \\
(ii):\ The L'Hopital's like rule for monotonicity, see \cite{pinelis2002hospital} or Lemma 2.2 in \cite{anderson1993inequalities}, states that a function $f/g$ on an interval $(a,b)$, satisfying $g'\neq 0$ and either $f(a)=0=g(a)$ or $f(b)=0=g(b)$, is strictly increasing/decreasing if $f'/g'$ is strictly increasing/decreasing on $(a,b)$.
For $f=|\log^m(u)|$ and $g=u-\log(u)-1,$ we have 
$$f'/g'=\frac{m\log(u)|\log^{m-2}(u)|}{u-1}$$ and for $\bar{f}=m\log(u)|\log^{m-2}(u)|$ and $\bar{g}=u-1,$ we obtain
$$\bar{f}'/\bar{g}'=\frac{(m-1)m|\log^{m-2}(u)|}{u}.$$
On $u\in(0,1),$ $\bar{f}'/\bar{g}'$ is strictly decreasing.
 Applying the L'Hopital's like rule for monotonicity twice yields the statement.
\end{proof}

\begin{proof}[Proof of Lemma \ref{lem.Hell_KL_bd}]
The inequality $\KL_2(P,Q)\leq \KL_{B}(P,Q)$ follows direct from the definition of the truncated Kullback-Leibler divergence.
Write $P=P^a+P^s$ for the Lebesgue decomposition of $P$ with respect to $Q$ such that $P^a \ll Q.$ The Lebesgue decomposition ensures existence of a set $A$ with $P^a(A)=0=P^s(A^c).$ For $x\in A,$ we define $dP/dQ(x):=+\infty.$
For the dominating measure $\mu=(P+Q)/2,$ denote by $p,p^a,p^s,q$ the $\mu$-densities of $P,P^a,P^s,Q,$ respectively. Since $p^sq=0,$
\begin{align*}
	H^2(P,Q) = \int \big(p^a +p^s - \sqrt{p^aq}\big) 
	\leq \int_{0<p^a/q\leq e^2} \big(p^a - \sqrt{p^aq}\big)+\int_{p^a/q> e^2} p^a + \int p^s.
\end{align*}
For every $u\in \mathbb{R},$ we have $1-u\leq e^{-u}$ and hence $e^u-1\leq ue^u.$ Substituting $u=\log(\sqrt{y})$ yields $\sqrt{y}-1 \leq \sqrt{y}\log(\sqrt{y})$ and therefore $y-\sqrt{y} \leq y\log(\sqrt{y})=y\log(y)/2.$ With $y=p^a/q,$ we find,
\begin{align*}
	H^2(P,Q) 
	\leq \int_{0<p^a/q\leq e^2} \frac{p^a}{2q} \log\Big(\frac{p^a}{q}\Big) q +\int_{p^a/q > e^2} p^a + \int p^s.
\end{align*}

The other direction works similarly. Second order Taylor expansion around one gives for $y>0,$ $y\log(y) \leq y -1+\tfrac12 (y-1)^2/(y\wedge 1).$ For $y=\sqrt{x},$ we find $x\log(x)=2\sqrt{x}\cdot \sqrt{x}\log(\sqrt{x})\leq 2(x-\sqrt{x})+(1\vee \sqrt{x})(\sqrt{x}-1)^2.$ Consequently, for each $B\geq 0,$
\begin{align*}
	\KL_B(P,Q)
	&= \int_{p^a/q \leq e^B}  \frac{p^a}{q} \log\Big(\frac{p^a}{q}\Big) q+B \int_{dP/dQ> e^B}\, dP \\
	&\leq  2e^{B/2}H^2(P,Q)+2\int_{p^a/q \leq e^B} p-\sqrt{pq}+B \int_{dP/dQ> e^B}\, dP.
\end{align*}
If $\int_{p^a/q \leq e^B} p^a-\sqrt{p^aq} \leq 0,$ we can use that $H^2(P,Q) \geq \tfrac 12 \int_{p/q\geq e^B} (\sqrt{p}-\sqrt{q})^2 \geq  \tfrac 12 \int_{p/q\geq e^B} p(1-e^{-B/2})^2$ and hence 
\begin{align*}
	\KL_B(P,Q)\leq 2\Big(e^{B/2}+(1-e^{-B/2})^{-2} \Big)H^2(P,Q).
\end{align*}
Otherwise, if $\int_{p^a/q \leq e^B} p^a-\sqrt{p^aq}> 0,$ we can upper bound 
\begin{align*}
	\KL_B(P,Q)
	&\leq  2e^{B/2}H^2(P,Q)+B(1-e^{-B/2})^{-1}\int_{p^a/q \leq e^B} p-\sqrt{pq}+B \int_{dP/dQ> e^B} \, dP \\
	&\leq  2e^{B/2}H^2(P,Q)+B(1-e^{-B/2})^{-1}\int p-\sqrt{pq} \\
	&= \Big(2e^{B/2}+B(1-e^{-B/2})^{-1}\Big) H^2(P,Q).
\end{align*}
The result now follows by observing that since $B\geq 2$, both $B(1-e^{-B/2})^{-1}$ and $2(1-e^{-B/2})^{-2}$ are less than $2e^{B/2}.$
\end{proof}

\begin{propositie}\label{P: LogSoftmax Lipschitz}
Recall that $\Phi$ denotes the softmax function. The function $\log(\bPhi(\cdot)):\mathbb{R}^K\rightarrow \mathbb{R}^K$ satisfies $|\log(\bPhi(\bx))-\log(\bPhi(\by))|_{\infty}\leq K\|\bx-\by\|_{\infty}$.
\end{propositie}
\begin{proof}
Consider the composition of the logarithm with the softmax function, that is,
\begin{equation*}
\begin{aligned}
\left(\log\left(\frac{e^{x_1}}{\sum_{j=1}^Ke^{x_j}}\right),\cdots,\log\left(\frac{e^{x_K}}{\sum_{j=1}^Ke^{x_j}}\right)\right).
\end{aligned}
\end{equation*}
It holds for $k,i\in\{1,\cdots,K\}$, $i\neq k$ that
\begin{equation*}
\begin{aligned}
    &\frac{\partial}{\partial x_k}\log\left(\frac{e^{x_k}}{\sum_{j=1}^Ke^{x_j}}\right)=1-\frac{e^{x_k}}{\sum_{j=1}^Ke^{x_j}},\\
&\frac{\partial}{\partial x_k}\log\left(\frac{e^{x_i}}{\sum_{j=1}^Ke^{x_j}}\right)=-\frac{e^{x_k}}{\sum_{j=1}^Ke^{x_j}}.
\end{aligned}
\end{equation*}
The partial derivatives are bounded in absolute value by one. The combined log-softmax function is therefore Lipschitz continuous (w.r.t to $\|\cdot\|_\infty$ norm for vectors) with Lipschitz constant bounded by $K.$
\end{proof}

\begin{proof}[Proof of Lemma \ref{L: Covering Number bound}]
We start proving the first bound. Notice that $g\in\log(\mF_{\bPhi}(L,\bfm,s))$ means that there exists a ReLU network $f_g\in \mF_{\id}(L,\bfm,s)$ such that $g(\mathbf{x})=\log(\bPhi(f_g(\bx)))$. By Lemma 5 of \cite{NonParametricRegressionReLU} it holds that $\mathcal{N}(\delta/(2K),\mF_{\id}(L,\bfm,s),\|\cdot\|_{\infty})\leq (4\delta^{-1}K(L+1)V^2)$.
Let $\delta>0$. Denote by $(\bf_j)_{j=1}^{\mathcal{N}_n}$ the centers of a minimal $\delta/(2K)$-covering of $\mF_{\id}(L,\bfm,s)$ with respect to $\|\cdot\|_{\infty}$. Triangle inequality gives that for each $\bf_j$ there exists a $\widehat{\bf_j}\in\mF_{\id}(L,\bfm,s)$ such that $(\widehat{\bf_j})_{j=1}^{\mathcal{N}_n}$ is an interior $\delta/K$-cover of $\mF_{\id}(L,\bfm,s)$. Let $\bg\in \log(\mF_{\bPhi}(L,\bfm,s))$. By the cover property, there exists a $j\in\{1,\cdots,\mathcal{N}_n\}$ such that $\|f_g-\widehat{f}_j\|\leq \delta/K$.
Proposition \ref{P: LogSoftmax Lipschitz} yields:
$$\|\bg-\log(\Phi(\widehat{\bf}_j))\|_{\infty}=\|\log(\Phi(\bf_g))-\log(\Phi(\widehat{\bf}_j))\|_{\infty}\leq K \|f_g-\widehat{f}_j\|_{\infty}\leq \delta.$$
Since $\bg\in \log(\mF_{\bPhi}(L,\bfm,s))$ was arbitrary and $\widehat{f}_j\in\mF_{\id}(L,\bfm,s)$ for $j=1,\cdots,\mathcal{N}_n$, this means that $(\log(\Phi(\widehat{\bf}_j))$ is an internal $\delta$-cover for $\log(\mF_{\bPhi}(L,\bfm,s))$ with respect to $\|\cdot\|_{\infty}$.
Hence
$$\mathcal{N}(\delta,\log(\mF_{\Phi}(L,\bfm,s)),\|\cdot\|_{\infty})\leq \mathcal{N}(\delta/(2K),\mF_{\id}(L,\bfm,s),\|\cdot\|_{\infty})\leq (4\delta^{-1}K(L+1)V^2).$$
Now we consider the second bound of the lemma. Using that $m_0=d$, $m_{L+1}=K$ and by removing inactive nodes, Proposition \ref{P: Inactive Node Removal}, we get that $m_{\ell}\leq s$ for $s=1,\cdots,L$, and thus
\begin{equation*}
    V\leq dKs^L2^{L+2}.
\end{equation*} Substituting this in the first bound and taking the logarithm yields the result.
\end{proof}

\end{document}